\documentclass[12pt,reqno]{amsart}

\usepackage{fixltx2e} 
\usepackage[utf8]{inputenc} 
\usepackage{amssymb, latexsym, stmaryrd, amsthm, dsfont, amsfonts, amsbsy, amsmath, mathrsfs} 
\usepackage{mathtools} 
\usepackage{bm, bbm} 
\usepackage{enumerate} 
\usepackage{verbatim} 
\usepackage{url}   
\usepackage{lscape} 
\usepackage{centernot}
\usepackage[dvipsnames]{xcolor}
\usepackage[colorinlistoftodos]{todonotes} 
\usepackage{nicefrac} 
\usetikzlibrary{calc}

\usepackage{microtype,tikz,tikz-cd}  
\makeatletter 
\def\MT@register@subst@font{\MT@exp@one@n\MT@in@clist\font@name\MT@font@list
 \ifMT@inlist@\else\xdef\MT@font@list{\MT@font@list\font@name,}\fi}
\makeatother

\makeatletter 
\newcommand{\myitem}[1]{%
\item[(#1)]\protected@edef\@currentlabel{#1}%
}
\makeatother

\usepackage[pdftex,bookmarks,bookmarksnumbered,linktocpage,   
         colorlinks,linkcolor=blue,citecolor=blue]{hyperref}
         
\usepackage[capitalize,noabbrev]{cleveref} 

\newcommand{\bit}{\begin{itemize}}    
\newcommand{\eit}{\end{itemize}}
\newcommand{\ben}{\begin{enumerate}}
\newcommand{\een}{\end{enumerate}}

\newcommand{\benroman}{\ben[\normalfont (i)]}  
\let\eroman\een

\newcommand{\bde}{\begin{description}}
\newcommand{\ede}{\end{description}}
\let\oper=\mathbb                               

\newcommand{\III}{\oper{I}}                     
\newcommand{\SSS}{\oper{S}}                     
\newcommand{\VVV}{\oper{V}}                     

\newcommand{\res}{{\upharpoonright}}

\theoremstyle{theorem}
\newtheorem{Theorem}{Theorem}[section]
\newtheorem{Theorem-n}{Theorem}

\newtheorem{Proposition}[Theorem]{Proposition}
\newtheorem{Modal Sahlqvist Theorem}[Theorem]{Modal Sahlqvist Theorem}
\newtheorem{Intuitionistic Sahlqvist Theorem}[Theorem]{Intuitionistic  Sahlqvist Theorem}
\newtheorem{Esakia Duality}[Theorem]{Esakia Duality}

\newtheorem{Main Lemma}[Theorem]{Main Lemma}
\newtheorem{Compactness Theorem}[Theorem]{Compactness Theorem}
\newtheorem{Los Theorem}[Theorem]{\LL o\'s' Theorem}
\newtheorem{Isbell Theorem}[Theorem]{Isbell's Zigzag Theorem}
\newtheorem{Diagram Lemma}[Theorem]{Diagram Lemma}

\newtheorem{Transfer Lemma}[Theorem]{Transfer Lemma}
\newtheorem{Subdirect Decomposition Theorem}[Theorem]{Subdirect Decomposition Theorem}

\newtheorem{Corollary}[Theorem]{Corollary}
\newtheorem{Claim}[Theorem]{Claim}

\theoremstyle{definition}
\newtheorem{Definition}[Theorem]{Definition}
\newtheorem{exa}[Theorem]{Example}

\theoremstyle{remark}

\newtheorem{Remark}[Theorem]{Remark}

\crefname{Theorem}{Theorem}{Theorems}
\crefname{Proposition}{Proposition}{Propositions}
\crefname{Lemma}{Lemma}{Lemmas}
\crefname{Corollary}{Corollary}{Corollaries}
\crefname{Claim}{Claim}{Claims}
\crefname{Definition}{Definition}{Definitions}
\crefname{exa}{Example}{Examples}
\crefname{Remark}{Remark}{Remarks}
\crefname{Fact}{Fact}{Facts}
\crefname{exer}{Exercise}{Exercises}
\crefname{problem}{Problem}{Problems}
\crefname{Compactness Theorem}{Compactness Theorem}{Compactness Theorems}
\crefname{Los Theorem}{\LL o\'s' Theorem}{\LL o\'s' Theorems}
\crefname{Isbell Theorem}{Isbell's Zigzag Theorem}{Isbell's Zigzag Theorems}
\crefname{Diagram Lemma}{Diagram Lemma}{Diagram Lemmas}
\crefname{Subdirect Decomposition Theorem}{Subdirect Decomposition Theorem}{Subdirect Decomposition Theorems}

\AddToHook{env/Theorem/begin}{\crefalias{Theorem}{Theorem}}
\AddToHook{env/Proposition/begin}{\crefalias{Theorem}{Proposition}}
\AddToHook{env/Lemma/begin}{\crefalias{Theorem}{Lemma}}
\AddToHook{env/Corollary/begin}{\crefalias{Theorem}{Corollary}}
\AddToHook{env/Claim/begin}{\crefalias{Theorem}{Claim}}
\AddToHook{env/Definition/begin}{\crefalias{Theorem}{Definition}}
\AddToHook{env/exa/begin}{\crefalias{Theorem}{exa}}
\AddToHook{env/Remark/begin}{\crefalias{Theorem}{Remark}}
\AddToHook{env/Compactness Theorem/begin}{\crefalias{Theorem}{Compactness Theorem}}
\AddToHook{env/Los Theorem/begin}{\crefalias{Theorem}{\LL o\'s' Theorem}}
\AddToHook{env/Isbell Theorem/begin}{\crefalias{Theorem}{Isbell's Zigzag Theorem}}
\AddToHook{env/Diagram Lemma/begin}{\crefalias{Theorem}{Diagram Lemma}}
\AddToHook{env/Subdirect Decomposition Theorem/begin}{\crefalias{Theorem}{Subdirect Decomposition Theorem}}

\let\leq=\leqslant
\let\nleq=\nleqslant
\let\geq=\geqslant 

\usepackage[mathscr]{euscript}
 \let\mathscr\relax 
\usepackage[scr]{rsfso}

\renewcommand{\int}{\mathsf{int}\,}

\bmdefine{\A}{A} 
\bmdefine{\C}{C}                                
                              
\bmdefine{\2}{2}
\bmdefine{\B}{B}
\bmdefine{\D}{D}
\bmdefine{\E}{E}

\bmdefine{\Term}{T} 
\bmdefine{\Free}{F}
\bmdefine{\Fb}{F}
\newcommand{\?}{\ensuremath{\mkern0.4\thinmuskip}}   

\usepackage{scalerel}

\newcommand{\K}{\mathsf{K}}
\newcommand{\M}{\mathsf{M}}

\newcommand{\HHH}{\mathbb{H}}
\newcommand{\PPP}{\mathbb{P}}
\newcommand{\QQQ}{\mathbb{Q}}

\newcommand{\ext}{\mathsf{ext}}
\newcommand{\extpp}{\mathsf{ext}_{\textsc{pp}}}

\newcommand{\imppp}{\mathsf{imp}_{pp}}

\let\LL\L 
\renewcommand{\L}{\mathscr{L}}

\newcommand{\F}{\mathcal{F}}

\DeclareUnicodeCharacter{001B}{}

\makeatletter
\renewenvironment{abstract}
  {%
    \small
    \begin{center}%
      {\bfseries \abstractname\par}%
    \end{center}%
  }
\makeatother

\begin{document}

\title{An addendum to ``The theory of implicit operations''}

\author{Luca Carai, Miriam Kurtzhals, and Tommaso Moraschini}

\address{Luca Carai: Dipartimento di Matematica ``Federigo Enriques'', Universit\`a degli Studi di Milano, via Cesare Saldini 50, 20133 Milano, Italy}\email{luca.carai.uni@gmail.com}

\address{Miriam Kurtzhals and Tommaso Moraschini: Departament de Filosofia, Facultat de Filosofia, Universitat de Barcelona (UB), Carrer Montalegre, $6$, $08001$ Barcelona, Spain}
\email{mkurtzku7@alumnes.ub.edu and tommaso.moraschini@ub.edu}

\date{\today}

\maketitle

\begin{abstract}
In this addendum to \cite{CKMIMPv3}, we provide a pair of counterexamples relevant to the theory of implicit operations. More precisely, we exhibit a pp expansion of a variety that fails to be a variety (although it is a quasivariety).\ Furthermore, we construct a sequence of varieties possessing a congruence preserving Beth companion, but no simple Beth companion.
\end{abstract}

\section{Introduction}

An $n$-ary \emph{partial function} on a set $X$ is a function $f \colon Y \to X$ for some $Y \subseteq X^n$. In this case, the set $Y$ will be called the \emph{domain} of $f$ and will be denoted by $\mathsf{dom}(f)$. This notion can be extended to classes of algebras as follows. An $n$-ary \emph{partial function} on a class of algebras $\mathsf{K}$ is a sequence $\langle f^\A : \A \in \mathsf{K}\rangle$, where $f^\A$ is an $n$-ary partial function on $A$ for each $\A \in \mathsf{K}$. By a \emph{partial function} on $\mathsf{K}$ we mean an $n$-ary partial function on $\mathsf{K}$ for some $n \in \mathbb{N}$. When $f$ is a partial function on $\mathsf{K}$ and $\A \in \mathsf{K}$, we denote the $\A$-component of $f$ by $f^\A$. Lastly, throughout this note by a \emph{formula} we mean a first order formula.

\begin{Definition}
A formula $\varphi(x_1, \dots, x_n, y)$ with $n \geq 1$ in the language of a class of algebras $\mathsf{K}$ is said to be \emph{functional} in $\mathsf{K}$ when for all $\A \in \mathsf{K}$ and $a_1, \dots, a_n \in A$ there exists at most one $b \in A$ such that $\A \vDash \varphi(a_1, \dots, a_n, b)$.
\end{Definition}

In other words, $\varphi$ is functional in $\mathsf{K}$ when
\[
\mathsf{K} \vDash (\varphi(x_1, \dots, x_n, y) \sqcap \varphi(x_1, \dots, x_n, z)) \to y \thickapprox z.
\]
In this case, $\varphi$ induces an $n$-ary partial function $\varphi^\A$ on each $\A \in \mathsf{K}$ with domain 
\[
\mathsf{dom}(\varphi^\A) = \{ \langle a_1, \dots, a_n \rangle \in A^n : \text{there exists $b \in A$ such that }\A \vDash \varphi(a_1, \dots, a_n, b) \},
\]
defined for every $\langle a_1, \dots, a_n \rangle \in \mathsf{dom}(\varphi^\A)$ as
$\varphi^\A(a_1, \dots, a_n) = b$, where $b$ is the unique element of $A$ such that $\A \vDash \varphi(a_1, \dots, a_n, b)$. 
Consequently,
\[
\varphi^{\mathsf{K}} = \langle \varphi^\A : \A \in \mathsf{K}\rangle
\]
is an $n$-ary partial function on $\mathsf{K}$. 

\begin{Definition}
An $n$-ary partial function $f$ on a class of algebras $\mathsf{K}$ is said to be
\benroman
\item \emph{defined by a formula $\varphi$} when $\varphi$ is functional in $\mathsf{K}$ and $f = \varphi^{\mathsf{K}}$;
\item \emph{implicit} when it is defined by some formula;
\item an \emph{operation} of $\mathsf{K}$ when for each homomorphism $h \colon \A \to \B$ with $\A,\B \in \mathsf{K}$ and $\langle a_1, \dots, a_n \rangle \in \mathsf{dom}(f^\A)$ we have $\langle h(a_1), \dots, h(a_n) \rangle \in \mathsf{dom}(f^\B)$ and
\[
h(f^\A(a_1,\dots,a_n))=f^\B(h(a_1),\dots,h(a_n));
\]
\item an \emph{implicit operation} of $\mathsf{K}$ when it is both implicit and an operation of $\mathsf{K}$. 
\eroman   
We denote the class of implicit operations of $\mathsf{K}$ by $\mathsf{imp}(\mathsf{K})$.
\end{Definition}

In elementary classes, implicit operations admit the following description (see \cite[Thm.~3.9]{CKMIMPv3}).

\begin{Theorem}
Let $f$ be a partial function on an elementary class $\mathsf{K}$. Then $f$ is an implicit operation of $\mathsf{K}$ if and only if it is defined by an existential positive formula. 
\end{Theorem}

\begin{exa}[\textsf{Monoids}]
A typical example of an implicit operation of the variety $\K$ of monoids arises from the idea of ``taking inverses''. More precisely, for every $\A \in \K$ let $f^\A$ be the unary partial function on $A$ with
\[
\mathsf{dom}(f^\A) = \{ a \in A : a \text{ is invertible}\}
\]
defined for every $a \in \mathsf{dom}(f^\A)$ as
\[
f^\A(a) = \text{the inverse of }a.
\]
Then $\langle f^\A : \A \in \K \rangle$ is an implicit operation of $\K$.
\qed
\end{exa}

Although concrete examples of implicit operations have long been known, the theory of implicit operations received its first systematic treatment in  \cite{CKMIMPv3}.
In this note, we exhibit two counterexamples relevant to the general theory of implicit operations. For this, we assume familiarity with the concepts and notation of \cite{CKMIMPv3}, as well as with the basics of the theory of Heyting algebras (see, e.g., \cite[Ch.~IX]{BD74}).

\section{A variety with a pp expansion that is a proper quasivariety}

Consider the linearly ordered Heyting algebra $\boldsymbol{C}_8$ with universe
\[
0 < a_1 < a_2 < \cdots < a_6 < 1.
\]
We consider the algebra $\boldsymbol{A}$ obtained by endowing $\boldsymbol{C}_8$ with a constant for the element $a_5$ as well as with a pair of binary operations $x + y$ and $x \ast y$ and a pair of unary operations $\Box x$ 
and $\Diamond x$ defined for every $a, b \in A$ as follows:
\begin{align*}
    a+b&= \begin{cases}
  a_6  & \text{if }a = 0 \text{ and }b \in \{ a_6, 1 \}; \\
  a_5  & \text{if }a = 0 \text{ and }b = a_3; \\
 a_2 & \text{if }(a=0  \text{ and }b \notin \{ a_3, a_6, 1 \})\text{ or }(a \ne 0 \text{ and }b \ne a_1);\\
 a_1 & \text{if } a \ne 0   \text{ and }b =a_1;
\end{cases}\\
a \ast b&= \begin{cases}
1  &\text{if } a = a_4  \text{ and } b = a_6; \\
0  & \text{otherwise};
\end{cases}
\\
\Box a &= \begin{cases}
1  & \text{if }a = a_5; \\
0  & \text{otherwise};
\end{cases}
\\
\Diamond a &= \begin{cases}
1  & \text{if }a \in \{ 0, a_6, 1 \}; \\
 a_1  & \text{if }a  \in \{ a_1, a_2 \}; \\
 a_3& \text{if }a \in \{ a_3, a_5 \};\\
 a_5 & \text{if }a = a_4.
\end{cases}
\end{align*}

Our aim is to prove the following. 

\begin{Theorem}\label{Thm: pp exp variety not variety}
The variety $\mathbb{V}(\A)$ has a pp expansion
 that is a proper quasivariety and is not congruence preserving.
\end{Theorem}

\begin{proof}
By \cite[Thm.~12.6]{CKMIMPv3} every congruence preserving pp expansion of a variety is a variety. So, it is sufficient to show that $\mathbb{V}(\A)$ has a pp expansion that is a proper quasivariety.
    The proof proceeds through a series of claims. First, 
observe that $A- \{ a_4 \}$ is the universe of a subalgebra $\A- \{ a_4 \}$ of $\A$.

\begin{Claim} \label{Claim : subalgebras of A} We have 
    $\oper{S}(\A) = \{\boldsymbol{A},  \boldsymbol{A} - \{ a_4 \}\}$.
\end{Claim}

\begin{proof}[Proof of the Claim]
As $\A- \{ a_4 \}$ is a subalgebra of $\A$, it suffices to prove the inclusion $\oper{S}(\A) \subseteq \{\boldsymbol{A},  \boldsymbol{A} - \{ a_4 \}\}$, which amounts to $\mathsf{Sg}^{\A}(\emptyset) = A \smallsetminus \{ a_4\}$.
First, observe that $\mathsf{Sg}^{\A}(\emptyset)$ contains the interpretations $0, a_5$, and $1$ of the constants. As 
\[
0+1 = a_6, \quad 1+0 = a_2, \quad \Diamond a_2 = a_1, \quad \text{ and }\quad \Diamond a_5 = a_3,
\]
we conclude that $\mathsf{Sg}^{\A}(\emptyset)$ contains $a_1, a_2, a_3$, and $a_6$ as well. Hence, $\mathsf{Sg}^{\A}(\emptyset) = A \smallsetminus \{ a_4\}$.
\end{proof}

\begin{Claim}\label{Claim : new claim on B chain SI}
   Let $\C \in \mathbb{V}(\A)$ be a finite nontrivial chain with second largest element $a$. Then $\C$ is subdirectly irreducible with monolith $\mathsf{Cg}^\C(a, 1)$.    
    \end{Claim}

\begin{proof}[Proof of the Claim]
It suffices to show that $\mathsf{Cg}^{\C}(1,a)$ is the monolith 
 of $\C$.  First, observe that $\mathsf{Cg}^{\C}(1,a) \in \mathsf{Con}(\C) - \{ \textup{id}_C\}$ because $a < 1$, where $1$ is the maximum of $\C$. Then consider $\theta \in \mathsf{Con}(\C) - \{ \textup{id}_C\}$. As $\theta \ne \textup{id}_C$, there exist distinct $b,c \in C$ such that $\langle b,c \rangle \in \theta$. Since $b \ne c$, we have $b \leftrightarrow b = 1$ and $b \leftrightarrow c \ne 1$, where $x \leftrightarrow y$ is a shorthand for $(x \to y) \land (y \to x)$.      As $a$ is the second largest element of $\C$, this implies $     (b \leftrightarrow b) \lor a = 1$ and $(b \leftrightarrow c) \lor a = a$.     Together with $\langle b, c \rangle \in \theta$, this yields $\langle 1, a \rangle \in \theta$, whence $\mathsf{Cg}^{\C}(1,a) \subseteq \theta$. 
\end{proof}

Observe that
\[
\theta =  \textup{id}_{A - \{ a_4 \}} \cup \{ \langle a_6, 1 \rangle, \langle 1, a_6 \rangle\}
\]
is a congruence of $\A - \{ a_4 \}$. Then let 
\[
\B = (\boldsymbol{A} - \{ a_4 \})/\theta.
\]

\begin{Claim}\label{Claim : the SI members of V(A) in pp counterexample}
We have $\mathbb{V}(\boldsymbol{A})_{\textsc{si}} = \oper{I}(\{\boldsymbol{A}, \boldsymbol{A} - \{ a_4 \}, \B\})$.
\end{Claim}

\begin{proof}[Proof of the Claim]
 Observe that all $\boldsymbol{A}, \boldsymbol{A} - \{ a_4 \}$, and $\B$ are finite nontrivial chains. Therefore, the inclusion from right to left follows from Claim \ref{Claim : new claim on B chain SI}. 

To prove the inclusion from left to right, observe that the variety $\mathbb{V}(\boldsymbol{A})$ is congruence distributive because it has a lattice reduct  (see, e.g., \cite[Thm.~7.2]{CKMIMPv3}). 
By J\'onsson's Theorem 
(see, e.g., \cite[Thm.~6.8]{BuSa00}) and \cite[Thm. 5.6(2)]{Ber11}
we have $\mathbb{V}(\boldsymbol{A})_{\textsc{si}}\subseteq \mathbb{H}\mathbb{S}(\boldsymbol{A})$.
Together with Claim \ref{Claim : subalgebras of A}, this yields  
\[
\mathbb{V}(\boldsymbol{A})_{\textsc{si}}\subseteq \mathbb{H}(\{ \boldsymbol{A}, \boldsymbol{A} - \{ a_4 \}\}).
\]
Therefore, to conclude the proof, it will be enough to show that $\A$ is simple and that, up to isomorphism, the only nontrivial homomorphic images of $\boldsymbol{A} - \{ a_4 \}$ are $\boldsymbol{A} - \{ a_4 \}$ and $\B$.

We begin by proving that $\A$ is simple,
which means that $\mathsf{Con}(\A)$ has exactly two elements. In view of Claim \ref{Claim : new claim on B chain SI}, it suffices to show that $\mathsf{Cg}^{\A}(a_6 ,1)
= A \times A$. 
Observe that $\langle 1,0 \rangle = \langle a_4 \ast a_6, a_4 \ast 1\rangle \in \mathsf{Cg}^{\A}(a_6,1)$. As the lattice reduct of $\A$ is a chain with extrema $0$ and $1$, this guarantees that $\mathsf{Cg}^{\A}(a_6,1) = A \times A$. 

Lastly, we prove that, up to isomorphism, the only nontrivial homomorphic images of $\boldsymbol{A} - \{ a_4 \}$ are $\boldsymbol{A} - \{ a_4 \}$ and $\B$. By the definition of $\B$ it will be enough to show that for every $\phi \in \mathsf{Con}(\boldsymbol{A} - \{ a_4 \})$,
\[
\phi \notin \{ \textup{id}_{A - \{ a_4 \}}, \theta \} \text{ implies }\phi = (A - \{ a_4 \}) \times (A - \{ a_4 \}).
\]
Consider $\phi \notin \{ \textup{id}_{A - \{ a_4 \}}, \theta \}$. Observe that the definition  of $\theta$ and Claim \ref{Claim : new claim on B chain SI} guarantee that $\theta \subsetneq \phi$. 
Therefore, there exist $c,d \in A - \{ a_4 \}$ such that $\langle c,d \rangle \in \phi - \theta$. From the definition of $\theta$ it follows that 
\[
c \ne d \, \, \text{ and }\, \, \{ c, d \} \ne \{ a_6, 1 \}.
\]
As $c \ne d$ and the lattice reduct of $\A - \{ a_4 \}$ is a chain, we can assume that $c < d$. 
From $c < d$, the right hand side of the above display, and the fact that $a_6$ is the second largest element of $A - \{ a_4 \}$ it follows that  $c < a_6$,   whence $c \leq a_5$. Consequently,
\[
\langle 1, a_5  \rangle = \langle a_5 \lor 1, a_5 \lor c \rangle = \langle a_5 \lor (c \to c), a_5 \lor (d \to c) \rangle \in \phi
\]
and, therefore, $\langle 1, 0 \rangle = \langle \Box a_5, \Box 1 \rangle \in \phi$. As before, this yields $\phi = (A - \{ a_4 \}) \times (A - \{ a_4 \})$. 
\end{proof}

Consider the pp formula 
\[
\varphi(x, y) = \exists z (x+y \thickapprox \Diamond z).
\]

\begin{Claim} \label{Claim : varphi in ext}
The formula    $\varphi(x, y)$ defines an extendable implicit operation $f$ of $\mathbb{V}(\boldsymbol{A})$ such that $f^\A$ is a total function defined for every $a \in A$ as
\[
f^{\boldsymbol{A}}(a)= \begin{cases}
a_1  & \text{if } a \ne 0; \\
a_3  & \text{if }a = 0.
\end{cases}\]
\end{Claim}

\begin{proof}[Proof of the Claim]
 We will show that $\varphi$ defines an extendable implicit operation $f$ of $\VVV(\A)$. The description of $f^\A$ in the statement will be an immediate consequence of our proof.  
 
In view of \cite[Cor.~8.14]{CKMIMPv3},
 it suffices to show that every member of $\VVV(\A)_\textsc{si}$ can be extended to one of $\VVV(\A)$ in which $\varphi(x, y)$ defines a total unary function. Recall from \cref{Claim : the SI members of V(A) in pp counterexample} that $\mathbb{V}(\boldsymbol{A})_{\textsc{si}} = \III(\{\boldsymbol{A}, {\boldsymbol{A} - \{ a_4 \}}, \B\})$. As $\boldsymbol{A} - \{ a_4 \} \leq \A$, we have $\mathbb{V}(\boldsymbol{A})_{\textsc{si}} \subseteq \III\SSS(\{ \A, \B\})$. Consequently, it will be enough to show that $\varphi(x, y)$ defines a total unary function both on $\A$ and $\B$.
    
We begin with the case of $\A$.\ We need to prove that for every $a \in A$ there exists a unique $b \in A$ such that $\A \vDash \varphi(a, b)$. To this end, consider $a \in A$. We have two cases: either $a = 0$ or $a \ne 0$. First, suppose that $a = 0$. Since
\[
a + a_3 = 0 + a_3 = a_5 = \Diamond a_4,
\]
the definition of $\varphi$ guarantees that $\A \vDash \varphi(a, a_3)$. Therefore, it only remains to show that $b = a_3$ for every $b \in A$ such that $\A \vDash \varphi(a, b)$. Consider $b \in A$ such that $\A \vDash \varphi(a, b)$. Then $a + b = \Diamond c$ for some $c \in A$. As $a = 0$, we have $a + b \in \{a_2, a_5, a_6\}$. Together with $\Diamond[A] = \{ a_1, a_3, a_5, 1\}$ and $a + b = \Diamond c$, this implies $a + b = a_5$ From the definition  of $+$ it thus follows that $b = a_3$, as desired.

Then we consider the case where $a \ne 0$. Since $a + a_1 = a_1 = \Diamond a_1$, 
the definition of $\varphi$ guarantees that $\A \vDash \varphi(a, a_1)$. 
Therefore, it only remains to show that $b = a_1$ for every $b \in A$ such that $\A \vDash \varphi(a, b)$. Consider $b \in A$ such that $\A \vDash \varphi(a, b)$. Then $a + b = \Diamond c$ for some $c \in A$. As $a \ne 0$, we have $a + b \in \{a_1, a_2 \}$. Together with $\Diamond[A] = \{ a_1, a_3, a_5, 1\}$ and $a + b = \Diamond c$, this implies $a + b = a_1$ From the definition of $+$ it thus follows that  $b = a_1$.

Next we consider the case of $\B = (\A - \{ a_4 \}) / \theta$. Since $\A - \{ a_4 \} \leq \A$ the definition of $\theta$ guarantees that for every $a, b \in A - \{ a_4 \}$,
\begin{align}
    \label{Eq : Miriam new equation}
     \begin{split}
    \B \vDash \varphi(a/\theta, b/\theta) \iff & \text{ there exists }c \in A - \{a_4\} \text{ such that }\\
    &\text{ either } a+^\A b = \Diamond^\A c \text{ or } \{ a+^\A b, \Diamond^\A c \} = \{ a_6, 1 \}. 
    \end{split}
\end{align}
Then let $a  \in  A - \{ a_4 \}$. As before, we have two cases: either $a = 0$ or $a \ne 0$. First, suppose that $a = 0$. Since
\[
a +^\A a_6 = 0 +^\A a_6 = a_6 \, \, \text{ and }\, \, \Diamond^\A a_6 = 1, 
\]
from $\langle a_6, 1 \rangle \in \theta$ it follows that 
\[
a / \theta +^\B a_6 / \theta = a_6 / \theta = 1 / \theta = \Diamond^{\B} a_6 / \theta.
\]
By the definition of $\varphi$ this guarantees that $\B \vDash \varphi(a / \theta, a_6 / \theta)$. Therefore, it only remains to show that $b / \theta = a_6 / \theta$ for every $b \in A - \{ a_4 \}$ such that $\B \vDash \varphi(a / \theta, b / \theta)$. Consider $b \in A - \{ a_4\}$ such that $\B \vDash \varphi(a / \theta, b / \theta)$. Let $c \in A - \{ a_4 \}$ be the element given by the right hand side of (\ref{Eq : Miriam new equation}). As $a =  0$, we have $a +^\A b \in \{ a_2, a_5, a_6 \}$. Together with $\Diamond c \in \Diamond[A - \{ a_4 \}] = \{ a_1, a_3, 1\}$, the right hand side of (\ref{Eq : Miriam new equation}) ensures that $a+^\A b = a_6$. By the definition of $+$ we obtain  $b \in \{ a_6, 1 \}$. As $\langle a_6, 1 \rangle \in \theta$, we conclude that $b / \theta = a_6 / \theta$, as desired. Then we consider the case where $a \ne 0$. In this case, $a +^\A a_1 = a_1 = \Diamond^\A a_1$. 
Therefore, $\B \vDash \varphi(a / \theta, a_1 / \theta)$ by the definition of $\varphi$. It only remains to show that $b / \theta = a_1 / \theta$ for every $b \in A - \{ a_4 \}$ such that $\B \vDash \varphi(a / \theta, b / \theta)$. Consider $b \in A - \{ a_4\}$ such that $\B \vDash \varphi(a / \theta, b / \theta)$. As before, let $c \in A - \{ a_4\}$ be the element given by right hand side of (\ref{Eq : Miriam new equation}). Since $a \ne  0$, we have $a +^\A b \in \{ a_1, a_2 \}$. Together with $\Diamond c \in \Diamond[A - \{ a_4 \}] = \{ a_1, a_3, 1\}$ and the right hand side of (\ref{Eq : Miriam new equation}), 
it follows that $a+^\A b = a_1$. By the definition of $+$ we obtain  $b = a_1$, whence $b / \theta = a_1 / \theta$.
\end{proof}

By \cref{Claim : varphi in ext} the formula $\varphi$ defines some $f \in \mathsf{ext}_{pp}(\mathbb{V}(\A))$. 
 Consider the $f$-expansion $\mathscr{L}_f$ of $\mathscr{L}_{\mathbb{V}(\A)}$
obtained by adding a new unary function symbol $g_f$ to $\mathscr{L}_{\mathbb{V}(\A)}$. Moreover, let $\mathsf{M}$ be the pp expansion $\SSS(\VVV(\A)[\mathscr{L}_\mathcal{F}])$ of $\mathbb{V}(\A)$ induced by 
$f$ and $\mathscr{L}_f$. 
To conclude the proof, it only remains to show that $\mathsf{M}$ is a proper quasivariety. 

First, $\mathsf{M}$ is a quasivariety by \cite[Thm.~10.3(ii)]{CKMIMPv3}.
We will prove that $\mathsf{M}$ is not a variety, i.e., it is not closed under $\HHH$. 
Recall from \cref{Claim : varphi in ext} that $f^\A$ is a total function. Therefore, the algebra $\A[\mathscr{L}_\mathcal{F}]$ is well defined. Moreover, the definition of $\A$ and the description of $f^\A$ in Claim \ref{Claim : varphi in ext} guarantee that $A - \{ a_4 \}$ is the universe of a subalgebra $\C$ of $\A[\mathscr{L}_\mathcal{F}]$.  Then 
from the definition of $\mathsf{M}$ it follows that
\[
\C \in \SSS(\A[\mathscr{L}_\mathcal{F}]) \subseteq \SSS(\mathbb{V}(\A)[\mathscr{L}_\mathcal{F}]) = \mathsf{M}.
\]
Now recall that
\[
\theta = \textup{id}_{A - \{ a_4 \}} \cup \{ \langle a_6, 1 \rangle, \langle 1, a_6 \rangle \}.
\]
As $\theta$ is a congruence of $\A - \{ a_4 \} = \C {\upharpoonright}_{\mathscr{L}_{\mathbb{V}(\A)}}$ which, moreover, is compatible with the new operation $g_f^\C = f^\A{\upharpoonright}_C$ by \cref{Claim : varphi in ext}, we obtain that $\theta$ is also a congruence of $\C$. We will prove that $\C / \theta \notin \mathsf{M}$. As $\C \in \mathsf{M}$, this will imply that $\mathsf{M}$ is not closed under $\HHH$, as desired. 

Suppose, with a view to contradiction, that $\C / \theta \in \mathsf{M}$. By the definition of $\mathsf{M}$ there exists $\D \in \mathbb{V}(\A)$ such that $f^\D$ is total and $\C / \theta \leq \D[\mathscr{L}_\mathcal{F}]$. Observe that 
\[ 
0 +^\A 1 = a_6 \, \, \text{ and } \, \, \Diamond^\A 1 = 1.
\]
Since $\langle a_6, 1 \rangle \in \theta$ and $\C{\upharpoonright}_{\mathscr{L}_{\mathbb{V}(\A)}} = \A - \{ a_4 \} \leq \A$, this yields
\[
0 +^{\C / \theta} 1 / \theta = a_6 / \theta = 1 / \theta = (\Diamond^\A 1) / \theta = \Diamond^{\C / \theta} 1 / \theta.
\]
Together with the definition of $\varphi$, this guarantees $\C / \theta \vDash \varphi(0 / \theta, 1 / \theta)$. Since $\varphi$ is a pp formula and $\C / \theta\leq \D[\mathscr{L}_\mathcal{F}]$, from 
\cite[Prop.~8.1]{CKMIMPv3} it follows that $\D[\mathscr{L}_\mathcal{F}] \vDash \varphi(0 / \theta, 1 / \theta)$. As $\varphi$ is a formula in $\mathscr{L}_{\mathbb{V}(\A)}$ and $\D = \D[\mathscr{L}_\mathcal{F}]{\upharpoonright}_{\mathscr{L}_{\mathbb{V}(\A)}}$, we obtain $\D \vDash \varphi(0 / \theta, 1 / \theta)$. Since $\varphi$ is the formula defining $f$ and $g_f^{\D[\mathscr{L}_\mathcal{F}]} = f^\D$, this yields
\[
g_f^{\D[\mathscr{L}_\mathcal{F}]}(0 / \theta) = f^\D(0 / \theta) = 1 / \theta.
\]
Therefore, $g_f^{\C / \theta}(0 / \theta) = 1 / \theta$ because $\C / \theta \leq \D[\mathscr{L}_\mathcal{F}]$.
On the other hand, we will prove that
\[
g_f^{\C / \theta} (0 / \theta) = g_f^{\C}(0) / \theta = g_f^{\A[\mathscr{L}_\mathcal{F}]}(0) / \theta  = f^\A(0) / \theta = a_3 / \theta \ne 1 / \theta,
\]
thus obtaining the desired contradiction. The first equality above holds by the definition of a quotient algebra, the second because $\C \leq \A[\mathscr{L}_\mathcal{F}]$, the third by the definition of $\A[\mathscr{L}_\mathcal{F}]$, and the fourth by \cref{Claim : varphi in ext}. Finally, the inequality at the end of the above display follows from the definition of $\theta$.
\end{proof}

\begin{Remark}
    The proof of \cref{Thm: pp exp variety not variety} yields that $\theta \in \mathsf{Con}(\C {\upharpoonright}_{\mathscr{L}_{\mathbb{V}(\A)}})-
    \mathsf{Con}_{\M}(\C)$, witnessing that the pp expansion $\M$ of $\VVV(\A)$ is not congruence preserving.
    \qed
\end{Remark}

\section{A nonsimple congruence preserving Beth companion}

While a generic pp expansion of a class of algebras $\mathsf{K}$ is of the form $\SSS(\mathsf{K}[\mathscr{L}_{\mathcal{F}}])$ for some $\mathcal{F} \subseteq \extpp(\mathsf{K})$, all the examples discussed in \cite{CKMIMPv3} are of the form $\mathsf{K}[\mathscr{L}_{\mathcal{F}}]$, i.e., $\mathsf{K}[\mathscr{L}_{\mathcal{F}}]$ is already closed under subalgebras. 

\begin{Definition}
    A pp expansion $\M$ of a class of algebras $\K$ is said to be
    \benroman
    \item  \emph{simple} when it is of the form $\K[\L_\F]$ for some $\mathcal{F} \subseteq \extpp(\K)$;
    \item a \emph{simple Beth companion} of $\K$ when it is simple and a Beth companion of $\K$.
    \eroman
\end{Definition}

Simple Beth companions have particularly nice properties.
For example, they are congruence preserving (see \cite[Thm.\ 12.3]{CKMIMPv3}) and often inherit or improve certain structural properties of the original class (see \cite[Thm.\ 12.7]{CKMIMPv3}). It is therefore natural to wonder whether every quasivariety $\mathsf{K}$ with a Beth companion has also a simple  Beth companion.  
This is the case, for instance, when $\K$ has 
the amalgamation property or has a Beth companion induced by implicit operations defined by conjunctions of equations (see \cite[Thm.~12.3]{CKMIMPv3}). Our aim is to show that the previous conjecture fails, 
even when $\K$ is a variety with a
congruence preserving Beth companion. Actually, a counterexample can be found among some of the simplest varieties of Heyting algebras, as we proceed to illustrate.

For every cardinal $\kappa$ let $\A_\kappa$ be the unique Heyting algebra whose lattice reduct is obtained by adding a new maximum $1$ to the powerset lattice $\langle \mathcal{P}(\kappa); \cap, \cup \rangle$. The implication of  $\A_\kappa$ is defined by the rule
\[
a \to b = \begin{cases}
  1  & \text{if }a \leq b; \\
 b  & \text{if }a = 1;\\
 (\kappa -a) \cup b & \text{if $a, b \in \mathcal{P}(\kappa)$ and $a \nleq b$}.
\end{cases}
\]
As expected, $\A_\kappa$ and the powerset Boolean algebra $\mathcal{P}(\kappa)$ are closely related, in the sense that $\mathcal{P}(\kappa)$ is isomorphic to the algebra obtained by quotienting $\A_\kappa$ under the congruence that glues $1$ with $\kappa$ and leaves any other element untouched.

The varieties generated by Heyting algebras of the form $\A_\kappa$ form the chain
\[
\VVV(\A_0) \subsetneq \VVV(\A_1)\subsetneq \dots \subsetneq \VVV(\A_n)\subsetneq \dots \subsetneq  \VVV(\A_\omega),
\]
where $\VVV(\A_\omega) = \VVV(\A_\kappa)$ for every infinite cardinal $\kappa$ (see \cite{HO70}).\footnote{Although we shall not rely on this fact, we remark that these are precisely the nontrivial varieties of Heyting algebras of depth $\leq 2$ (see also \cite[Exa.~10.18]{CKMIMPv3}).}

 The remainder of this section is devoted to showing that for $n \geq 3$ the variety $\mathbb{V}(\A_n)$ 
 provides a counterexample to the conjecture that every 
 congruence preserving Beth companion of a quasivariety is  simple. More precisely, we will establish the next result.

\begin{Theorem}\label{Thm : proper Beth companion}
The following conditions hold for every $\kappa \in \mathbb{N} \cup \{ \omega \}$:
\benroman
\item $\VVV(\A_\kappa)$ has a   congruence preserving   Beth companion;
\item $\VVV(\A_\kappa)$ has a  simple Beth companion if and only if $\kappa \in \{ 0, 1, 2, \omega \}$.
\eroman
\end{Theorem}

It is known that $\VVV(\A_0), \VVV(\A_1), \VVV(\A_2)$, and $\VVV(\A_\omega)$ have the strong epimorphism surjectivity property (see \cite[Thm.~8.1]{Mak00}). Consequently, they are their own Beth companions by \cite[Thm.~11.9(vi)]{CKMIMPv3}. 
When viewed as Beth companions of themselves, they are obviously 
 congruence preserving and simple. Hence, in order to prove \cref{Thm : proper Beth companion}, it will be enough to establish the following.

\begin{Theorem}\label{Thm : proper Beth companion : main}
For every $n \geq 3$ the variety $\VVV(\A_n)$ has a congruence preserving Beth companion but no simple Beth companion.
\end{Theorem}

The proof of  \cref{Thm : proper Beth companion : main} proceeds through a series of observations. First,
 if an algebra $\A$ has a lattice reduct, then $\VVV(\A)$ is congruence distributive (see, e.g., \cite[Thm.~7.2]{CKMIMPv3}). Therefore, the following is an immediate consequence of 
 a version of J\'onsson’s Theorem for finitely subdirectly irreducible algebras (see, e.g., \cite[Thm.~2.12]{CKMIMPv3} and \cite[Thm. 5.6(2)]{Ber11}).

\begin{Proposition}\label{Prop : Jonsson lattice : easy}
Let $\A$ be a finite algebra with a lattice reduct. Then $\VVV(\A)_\textsc{fsi} \subseteq \HHH\SSS(\A)$.
\end{Proposition}

As an application of Proposition \ref{Prop : Jonsson lattice : easy}, we obtain a transparent description of $\VVV(\A_n)_\textsc{fsi}$.

\begin{Proposition}\label{Prop : proper Beth completion : FSI members of V(An)}
For every $n \in \mathbb{N}$ we have $\VVV(\A_n)_\textsc{fsi} = \III( \A_0, \dots, \A_n) = \III\SSS(\A_n)$.
\end{Proposition}

\begin{proof}
By Proposition \ref{Prop : Jonsson lattice : easy} we have $\VVV(\A_n)_\textsc{fsi} \subseteq \HHH\SSS(\A_n)$. Moreover, by inspection it is possible to check that (up to isomorphism) the finitely subdirectly irreducible members of $\HHH\SSS(\A_n)$ are  $\A_0, \dots, \A_n$.
Together with $\VVV(\A_n)_\textsc{fsi} \subseteq \HHH\SSS(\A_n) \subseteq \VVV(\A_n)$, this yields $\VVV(\A_n)_\textsc{fsi} = \III( \A_0, \dots, \A_n )$. Lastly, the equality $\III( \A_0, \dots, \A_n) = \III\SSS(\A_n)$ is an immediate consequence of the definition of $\A_0, \dots, \A_n$.
\end{proof}

\begin{Corollary}\label{Cor : proper Beth completion : V(An) = Q(An)}
For every $n \in \mathbb{N}$ we have $\VVV(\A_n) = \QQQ(\A_n)$.
\end{Corollary}

\begin{proof}
From the Subdirect Decomposition Theorem (see, e.g., \cite[Thm.~8.6]{BuSa00}) and \cref{Prop : proper Beth completion : FSI members of V(An)} it follows that
\[
\VVV(\A_n) = \III\SSS\PPP(\VVV(\A_n)_\textsc{fsi}) = \III\SSS\PPP\III\SSS(\A_n) \subseteq \QQQ(\A_n).
\]
Since the inclusion $\QQQ(\A_n) \subseteq \VVV(\A_n)$ always holds, we conclude that $\VVV(\A_n) = \QQQ(\A_n)$.
\end{proof}

We will make use of the following properties   typical of the Heyting algebras of the form $\A_\kappa$ for a cardinal $\kappa$. 
As all of them are immediate consequences of the definition of $\A_\kappa$, their proof is omitted. First, observe that $\A_\kappa$ has a second largest element (namely, $\kappa$) that we denote by $e$. For every $a, b \in A_\kappa$ we have
\begin{align}
a \lor b = 1 &\iff a = 1 \text{ or }b = 1;\label{Eq : tricks for An : Beth completion : 1}\\
0 < a \leq e & \iff a \lor \lnot a = e ;\label{Eq : tricks for An : Beth completion : 2}\\
a \in \{ 0, e, 1 \} & \iff \lnot \lnot a = 1;\label{Eq : tricks for An : Beth completion : 3}\\
(a \ne e \text{ or } a=0) & \iff \lnot \lnot a = a.\label{Eq : tricks for An : Beth completion : 3b}
\end{align}
We recall that an element $a$ of an algebra $\B$ with a bounded lattice reduct is an \emph{atom} when $a \ne 0$ and there exists no $b \in B$ such that $0 < b < a$. 
 To simplify notation, we will make use of the following shorthands for every algebra $\B$ with a bounded lattice reduct and $a \in B$:
\begin{align*}
\mathsf{at}(\B) &= \{ b \in B : b \text{ is an atom of $\B$}\};\\
\mathsf{at}_\B(a) &= \{ b \in \mathsf{at}(\B) : b \leq a\}.
\end{align*}
Moreover, for every $\B \leq \A_n$ and $a \in  B$ the following holds:
\begin{align}
&\text{if } a \neq 1 
 \text{ then } a = \bigvee \mathsf{at}_{\B}(a);
\label{Eq : tricks for An : Beth completion : 4}\\
\text{if $b \in \mathsf{at}(\B)$, then}& \text{ either ($b \leq a$ and $b \nleq \lnot a$) or ($b \nleq a$ and $b \leq \lnot a$)}.\label{Eq : tricks for An : Beth completion : 5}
\end{align}
We also rely on the following properties that hold in every Heyting algebra. First, for every $a_1, \dots, a_m \in A_\kappa$,
\begin{equation}
\bigwedge_{i = 1}^{m} \lnot a_i = 1 \iff a_i = 0 \text{ for every }i \leq m.
\label{Eq : tricks for An : Beth completion : 6}
\end{equation}
Second, for every $a, b \in A_\kappa$,
\begin{align}
a \leq b &\iff a \to b = 1;\label{Eq : tricks for An : Beth completion : 7}\\
a \leq b & \, \, \, \Longrightarrow \lnot \lnot a \leq \lnot \lnot b.\label{Eq : tricks for An : Beth completion : 8}
\end{align}

Now, fix  $n \geq 3$. For each positive $m \leq n-1$ let $s_{m,n}$ and $d$ denote the terms
\[
s_{m, n} = \bigvee_{i = 1}^{n+1} z_i^m \, \, \text{ and }\, \,  d = x \lor \lnot x,
\]  
where $x, z_1^m, \dots, z_{n+1}^m$ are variables.
Then let $\psi_{m, n}(x, y, z_1^m, \dots, z_{n+1}^m)$ be the conjunction of the following formulas:
\begin{align*}
&\bigsqcap_{i = 1}^{n+1} (d(x) \thickapprox d(z_i^m));\\  
d(x&) \lor \lnot \lnot (x \lor s_{m, n}) \thickapprox y;\\
\Big((s_{m, n} \to x) \land \bigwedge^{m+1}_{\substack{ i, j = 1 \\ i \ne j}}\lnot  (z_i^m \land & z_j^m)\Big) \lor \Big((s_{m, n} \to \lnot x) \land \bigwedge^{n+1}_{\substack{i, j = m+2 \\ i \ne j}}\lnot (z_i^m \land z_j^m)\Big) \thickapprox 1.
\end{align*}
For each positive $k \leq n-1$ let $\gamma_{k, n}(x, y, z_1^1, \dots, z_{n+1}^1, \dots, z_1^k, \dots, z_{n+1}^k, w_1, \dots, w_k)$ be the formula
\[
\Big(y \thickapprox \bigvee_{m = 1}^k w_m\Big) \sqcap  \bigsqcap_{m = 1}^k \psi_{m, n}(x, w_m, z_1^m, \dots, z_{n+1}^m)
\]
and define
\[
\varphi_{k, n}(x, y) = \exists z_1^1, \dots, z_{n+1}^1, \dots, z_1^k, \dots, z_{n+1}^k, w_1, \dots, w_k \gamma_{k, n}.
\]
Observe that $\varphi_{k, n}(x, y)$ is a pp formula for every $n \geq 3$ and positive $k \leq n-1$. We will prove the following.

\begin{Proposition}\label{Prop : proper Beth completion : varphi is functional}
For every $n \geq 3$, positive $k \leq n-1$, and $a, b \in A_n$,
\begin{align*}
\A_n \vDash \varphi_{k, n}(a, b) \iff &  \text{ either }( a \in \{ 0, e, 1 \} \text{ and }b = 1)\\
& \text{ or } (0 < a < e \text{ and }b = 1 \text{ and the number of atoms below $a$ is $\leq k$})\\
&\text{ or }(0 < a < e = b \text{ and the number of atoms below $a$ is $\geq k+1$}).
\end{align*}
\end{Proposition}

\begin{proof}
We begin by proving the implication from left to right. Suppose that $\A_n \vDash \varphi_{k, n}(a, b)$. Then there exist $c_1^1, \dots, c_{n+1}^1, \dots, c_1^k, \dots, c_{n+1}^k, d_1, \dots, d_k \in A_n$ such that 
\begin{equation}
b = \bigvee_{m = 1}^k d_m\label{Eq : proper Beth completion : total 1}
\end{equation}
and for every positive $m \leq k$ both
\begin{align}
a \lor \lnot a &= c_1^m \lor \lnot c_1^m = \dots = c_{n+1}^m \lor \lnot c_{n+1}^m;\label{Eq : proper Beth completion : total 2}\\
d_m &= a \lor \lnot a \lor \lnot \lnot \Big(a \lor \bigvee_{i = 1}^{n+1} c_i^m\Big)\label{Eq : proper Beth completion : total 3}
\end{align}
and 
\[
1  =\Big(\Big(\bigvee_{i = 1}^{n+1}c_i^m \to a\Big) \land \bigwedge^{m+1}_{\substack{ i, j = 1 \\ i \ne j}}\lnot (c_i^m \land c_j^m)\Big) \lor \Big(\Big(\bigvee_{i = 1}^{n+1}c_i^m \to \lnot a\Big) \land \bigwedge^{n+1}_{\substack{i, j = m+2 \\ i \ne j}}\lnot (c_i^m \land c_j^m)\Big).
\]
Together with (\ref{Eq : tricks for An : Beth completion : 1}), (\ref{Eq : tricks for An : Beth completion : 6}), and (\ref{Eq : tricks for An : Beth completion : 7}), the above display yields that for every $m \leq k$,
\begin{equation}\label{Eq : proper Beth completion : total 4}
\begin{split}
\text{either }&\Big(\bigvee_{i = 1}^{n+1}c_i^m \leq a \text{ and } c_i^m \land c_j^m = 0 \text{ for all distinct $i, j$ with   $1 \leq	 i, j \leq m+1$}\Big) \\
\text{ or } &\Big(\bigvee_{i = 1}^{n+1}c_i^m \leq \lnot a \text{ and } c_i^m \land c_j^m = 0 \text{ for all distinct $i, j$ with   $m+2 \leq	 i, j \leq n+1$}\Big).
\end{split}
\end{equation}

By the definition of $\A_n$ we have two cases: either $a \in \{ 0, e, 1 \}$ or $0 < a < e$. First, suppose that $a \in \{ 0, e, 1 \}$. We need to prove that $b = 1$. To this end, observe that for every $m \leq k$,
\[
\lnot \lnot a \leq \lnot \lnot \Big(a \lor \bigvee_{i = 1}^{n+1} c_i^m\Big) \leq a \lor \lnot a \lor \lnot \lnot \Big(a \lor \bigvee_{i = 1}^{n+1} c_i^m\Big) = d_m,
\]
where the first inequality holds by (\ref{Eq : tricks for An : Beth completion : 8}), the second is straightforward, and the last equality by (\ref{Eq : proper Beth completion : total 3}). Since $a \in \{ 0, e, 1 \}$, we have $\lnot \lnot a = 1$ by (\ref{Eq : tricks for An : Beth completion : 3}). Together with the above display, we obtain $d_m = 1$ for every $m \leq k$. By (\ref{Eq : proper Beth completion : total 1}) we conclude that $b = 1$, as desired.

Next, we consider the case where $0 < a < e$. In this case,  $a \lor \lnot a = e$ by (\ref{Eq : tricks for An : Beth completion : 2}). Therefore, from (\ref{Eq : proper Beth completion : total 2}) it follows that $c_i^m \lor \lnot c_i^m = e$ for all positive $m \leq k$ and $i \leq n+1$. By (\ref{Eq : tricks for An : Beth completion : 2}) this yields
\begin{equation}\label{Eq : proper Beth completion : total 5}
0 < c_i^m  \text{ for all positive }m \leq k \text{ and } i \leq n+1.
\end{equation}

We have two subcases: either the number of atoms below $a$ is $\leq k$ or $\geq k+1$. First, suppose that it is $p \leq k$. We need to prove that $b = 1$. As $\A_n$ has $n$ atoms by definition, the number of atoms below $\lnot a$ is $n-p$ by (\ref{Eq : tricks for An : Beth completion : 5}). From (\ref{Eq : proper Beth completion : total 4}) in the case where $m = p$  it follows that 
\begin{align*}
\text{either }&\Big(\bigvee_{i = 1}^{n+1}c_i^p \leq a \text{ and } c_i^p \land c_j^p = 0 \text{ for all distinct $i, j$ with   $1 \leq	 i, j \leq p+1$}\Big) \\
\text{ or } &\Big(\bigvee_{i = 1}^{n+1}c_i^p \leq \lnot a \text{ and } c_i^p \land c_j^p = 0 \text{ for all distinct $i, j$ with   $p+2 \leq	 i, j \leq n+1$}\Big).
\end{align*}
The right hand side of the first line of the above display implies that the sets of atoms below each of the $c_i^p$ for $1 \leq i \leq p+1$ must be pairwise disjoint.
Moreover, observe that $\A_n$ is finite and, therefore, each nonzero element is above an atom. Together with \eqref{Eq : proper Beth completion : total 5}, this implies that there is at least one atom below each $c_i^p$. Consequently, there must be at least $p+1$ distinct atoms below the join of $c_{1}^p, \dots, c_{p+1}^p$. 
Together with the left hand side of the first line of the above display, this implies that the number of atoms below $a$ is $\geq p+1$, which is false by assumption. Therefore, 
\[
\bigvee_{i = 1}^{n+1}c_i^p \leq \lnot a \text{ and } c_i^p \land c_j^p = 0 \text{ for all distinct $i, j$ with   $p+2 \leq	 i, j \leq n+1$}.
\]
As before, the right hand side of the above display and (\ref{Eq : proper Beth completion : total 5}) imply that the number of distinct atoms below the join of $c_{p+2}^p, \dots, c_{n+1}^p$ must be at least $n-p$.
  Observe that by the left hand side of the above display and \eqref{Eq : tricks for An : Beth completion : 4} it follows that every atom below the join of $c_{p+2}^p, \dots, c_{n+1}^p$ must be also below $\lnot a$. As by assumption the number of atoms below $\lnot a$ is precisely $n-p$, the set of atoms below $\lnot a$ must coincide with the set of atoms below $c_{p+2}^p \lor \dots \lor c_{n+1}^p$. Therefore, using \eqref{Eq : tricks for An : Beth completion : 4} we obtain  \[
\lnot a = \bigvee_{i= 1}^{n+1}c_i^{p}.
\] 
Together with 
(\ref{Eq : proper Beth completion : total 3}), this yields
\[
 a \lor \lnot a \lor \lnot \lnot (a \lor \lnot a ) 
 = 
 a \lor \lnot a \lor \lnot \lnot (a \lor \bigvee_{i= 1}^{n+1}c_i^{p}) = d_p.
\]
As $0 < a < e$ by assumption, from (\ref{Eq : tricks for An : Beth completion : 2}) and (\ref{Eq : tricks for An : Beth completion : 3}) it follows that $\lnot \lnot (a \lor \lnot a) = \lnot \lnot e = 1$. Therefore, the above display yields
\[
1 = a \lor \lnot a \lor 1 = a \lor \lnot a \lor \lnot \lnot (a \lor \lnot a )  
  =   d_p.
\]
By (\ref{Eq : proper Beth completion : total 1}) we conclude that $b = 1$, as desired.
It only remains to consider the case where the number of atoms below $a$ is $\geq k+1$. We need to prove that $b = e$. As $\A_n$ has $n$ atoms by definition, the number of atoms below $\lnot a$ is $\leq n-k -1$ by (\ref{Eq : tricks for An : Beth completion : 5}). Then consider a positive $m \leq k$. Since $n-k-1 < n- m$,
the number of atoms below $\lnot a$ is $< n-m$. Since (\ref{Eq : proper Beth completion : total 5}) and the second line of (\ref{Eq : proper Beth completion : total 4}) would imply that the number of atoms below $\lnot a$ is $\geq n-m$, we conclude that the first line of (\ref{Eq : proper Beth completion : total 4}) holds. Consequently, 
\begin{equation}\label{Eq : proper Beth completion : definition of phi : qwert}
\bigvee_{i = 1}^{n+1}c_i^m \leq a.
\end{equation}
We will prove that the following holds:
\[
e = a \lor \lnot a \leq a \lor \lnot a \lor \lnot \lnot \Big(a \lor \bigvee_{i= 1}^{n+1}c_i^{m}\Big) \leq a \lor \lnot a \lor \lnot \lnot (a \lor a) = a \lor \lnot a = e.
\]
The first and the last steps above hold by $0 < a \leq e$ and (\ref{Eq : tricks for An : Beth completion : 2}), the second is straightforward, the third by (\ref{Eq : proper Beth completion : definition of phi : qwert}) and (\ref{Eq : tricks for An : Beth completion : 8}), and the fourth by $a = \lnot \lnot a$, which follows from  $a \ne e$ 
and (\ref{Eq : tricks for An : Beth completion : 3b}).
Together with (\ref{Eq : proper Beth completion : total 3}), the above display yields $d_m = e$. As this holds for every $m \leq k$, from (\ref{Eq : proper Beth completion : total 1}) it follows that $b = e$, as desired.

Next we prove the implication from right to left in the statement. 
Recall from the definition of $\varphi_{k,n}$ that it suffices to find $c_i^m, d_m$ for $i \leq n+1$ and $m \leq k$ such that 
\begin{equation} \label{Eq : to show psi}
    \A_n\vDash \Big(b \thickapprox \bigvee_{m = 1}^k  d_m \Big) \sqcap \bigsqcap_{m = 1}^k \psi_{m,n}(a,  d_m, c_1^m, \dots, c_{n+1}^m).
\end{equation}
First, suppose that $a \in \{0,1\}$. In this case, $b = 1$ by assumption.  
Choose $c_i^m = 0$ and $d_m= 1$ 
for all $i \leq n+1$ and $m \leq k$. 
Clearly, we have
\[
b =1 = \bigvee_{m = 1}^k  d_m.
\]
From
\eqref{Eq : tricks for An : Beth completion : 1} it follows that for each $m \leq k$  we have $d(a) = 1$ and, therefore,
\begin{align*}
&d(a) = 1 = d(0) = d(c_i^m) \text{  for each } i \leq n+1 \text{ and }\\ 
&d(a) \lor \lnot \lnot \Big(a \lor \bigvee_{i = 1}^{n+1} c_i^m\Big) = 1 \lor \lnot\lnot a = 1  = d_m,
\end{align*}
which proves the validity of the first two conjuncts of $\psi_{m,n}$. Moreover, it holds that 
\begin{align*}
\Big(\Big(&\bigvee_{i = 1}^{n+1} c_i^m \to a\Big)  \land \bigwedge^{m+1}_{\substack{ i, j = 1 \\ i \ne j}}\lnot  (c_i^m \land c_j^m)\Big) \lor \Big(\Big(\bigvee_{i = 1}^{n+1} c_i^m \to \lnot a\Big) \land \bigwedge^{n+1}_{\substack{i, j = m+2 \\ i \ne j}}\lnot (c_i^m \land c_j^m)\Big)\\
=&\Big((0 \to a) \land \bigwedge^{m+1}_{\substack{ i, j = 1 \\ i \ne j}}\lnot  0\Big) \lor \Big((0 \to \lnot a) \land \bigwedge^{n+1}_{\substack{i, j = m+2 \\ i \ne j}}\lnot 0\Big) = 1.
\end{align*}
This establishes \eqref{Eq : to show psi} for the case where $a \in \{0,1\}$.

It only remains to consider the case where  $0 < a < 1$. Observe that choosing $c_i^m \in \mathsf{at}(\A_n)$ for all $i \leq n+1$ and $m \leq k$ guarantees that 
\begin{equation} \label{Eq : first psi eq}
d(a) = e = d(c_i^m) \text{ for all } i \leq n+1 \text{ and } m \leq k
\end{equation} by \eqref{Eq : tricks for An : Beth completion : 2}.
Moreover,  \eqref{Eq : tricks for An : Beth completion : 4} implies that, in order to guarantee that 
\[
\Big(\Big(\bigvee_{i = 1}^{n+1} c_i^m \to a\Big) \land \bigwedge^{m+1}_{\substack{ i, j = 1 \\ i \ne j}}\lnot  (c_i^m \land c_j^m)\Big) \lor \Big(\Big(\bigvee_{i = 1}^{n+1} c_i^m \to \lnot a\Big) \land \bigwedge^{n+1}_{\substack{i, j = m+2 \\ i \ne j}}\lnot (c_i^m \land c_j^m)\Big) = 1,
\] 
it suffices to choose $c_i^m$ so that one of the following holds:
\begin{align} 
    &\{c_1^{m}, \dots, c_{n+1}^m\} = \mathsf{at}_{\A_n}(a) \text{ and } c_i^m \neq c_j^m \text{ for all }  i,j\in \{1, \dots, m+1\} \text{ with } i \neq j, \label{Eq : long equation conditions 1}\\
    &\{c_1^{m}, \dots, c_{n+1}^m\} = \mathsf{at}_{\A_n}(\lnot a) \text{ and } c_i^m \neq c_j^m \text{ for all } i ,j \in \{m+2, \dots, n+1\} \text{ with } i \neq j.\label{Eq : long equation conditions 2}
\end{align}

We distinguish three cases. First, let $a = e$. Then $b = 1$ by assumption. Choose $c_i^m \in \mathsf{at}(\A_n) = \mathsf{at}_{\A_n}(e)$ for all $i \leq n+1$ and $m \leq k$ such that $\{c_1^m, \dots, c_n^m\}$ are precisely the $n$ distinct atoms of $\A_n$ and let $d_m  = 1$ for each $m \leq k$. Then condition \eqref{Eq : long equation conditions 1} is satisfied, since $m \leq k \leq n-1$, and thus $m+1 \leq n$. 
Therefore, to verify \eqref{Eq : to show psi}, it only remains to observe that for each $m \leq k$ we have
\[
d(a) \lor \lnot \lnot \Big( a \lor \bigvee_{i = 1}^{n+1} c_i^m\Big) = d(e) \lor \lnot \lnot (a \lor e) = d(e) \lor \lnot \lnot e = 1 =  d_m,
\]
which is true by  \eqref{Eq : tricks for An : Beth completion : 3} and $a = e$.

Next we consider the case where $0 < a < e$ and $|\mathsf{at}_{\A_n}(a)| = p \leq k$. Then $b = 1$ by assumption.
For all $m < p$ and $i \leq n+1$ consider $c_i^m \in \mathsf{at}_{\A_n}(a)$ such that $\{c_{1}^m, \dots, c_{p}^m\} = \mathsf{at}_{\A_n}(a)$ and $ d_m  = e$. 
Moreover, for all $p \leq m \leq k$ and $i \leq n+1$ consider $c_i^m \in \mathsf{at}_{\A_n}(\lnot a)$ such that $\{c_{p+2}^m, \dots, c_{n+1}^m\} = \mathsf{at}_{\A_n}(\lnot a)$ and $d_m = 1$. 
Then for $m < p$ condition \eqref{Eq : long equation conditions 1} is satisfied and by \eqref{Eq : tricks for An : Beth completion : 2},   (\ref{Eq : tricks for An : Beth completion : 3}), \eqref{Eq : tricks for An : Beth completion : 3b}, and $0 < a < e$ we have 
\[
d(a) \lor \lnot \lnot \Big( a \lor \bigvee_{i = 1}^{n+1} c_i^m\Big) = d(a) \lor \lnot \lnot ( a \lor a) = e =  d_m.
\]
On the other hand, for every $m$ such that $p \leq m\leq k$ condition \eqref{Eq : long equation conditions 2} is satisfied. Moreover, using \eqref{Eq : tricks for An : Beth completion : 2}, \eqref{Eq : tricks for An : Beth completion : 3}, and $0 < a < e$, we obtain 
\[
d(a) \lor \lnot \lnot \Big( a \lor \bigvee_{i = 1}^{n+1} c_i^m\Big) = d(a) \lor \lnot \lnot ( a \lor \lnot a) = e \lor \lnot \lnot e = 1 =  d_m.
\]
Since $1 = \bigvee_{m \leq k}  d_m$ (because $d_k = 1$), this verifies that \eqref{Eq : to show psi} holds.

It only remains to consider the case where $0 < a < e$ and $|\mathsf{at}_{\A_n}(a)| = p \geq k+1$. In this case, we have $b = e$ by assumption. Then for all $i \leq n+1$ and $m \leq k$ consider 
 $c_i^m \in \mathsf{at}_{\A_n}(a)$  such that $\{c_{1}^{m}, \dots, c_{p}^{m}\} = \mathsf{at}_{\A_n}(a)$. Also choose $ d_m = e$ for each $m \leq k$.
Then  \eqref{Eq : long equation conditions 1} it satisfied because $m + 1 \leq k +1 \leq p$. 
Therefore, to conclude the proof of \eqref{Eq : to show psi}, it only remains to observe that for each $m \leq k$ we have
\[
d(a) \lor \lnot \lnot \Big( a \lor \bigvee_{i = 1}^{n+1} c_i^m\Big) = d(a) \lor \lnot \lnot ( a \lor a) = e = d_m,
\]
which holds by 
\eqref{Eq : tricks for An : Beth completion : 2},  
\eqref{Eq : tricks for An : Beth completion : 3b}, and $0 < a < e$. 
This completes the proof.
\end{proof}

As a consequence of Proposition \ref{Prop : proper Beth completion : varphi is functional}, we get the following. 

\begin{Corollary}\label{Cor : proper Beth completion : fnk is extendable}
For every $n \geq 3$ and positive $k \leq n-1$ the formula $\varphi_{k, n}$ defines an implicit operation $f_{k, n} \in \ext_{pp}(\VVV(\A_n))$ such that $f_{k,n}^{{\A}_n}$ is total and for every $a \in A_n$,
\[
 f_{k,n}^{{\A}_n}(a)= \begin{cases}
  1  & \text{if }a \in \{ 0, e, 1 \}; \\
    1  & \text{if }0 < a < e  \text{ and } |\mathsf{at}_{\A_n}(a)| \leq k;\\ 

    e  & \text{if }0 < a < e  \text{ and } |\mathsf{at}_{\A_n}(a)| \geq k+1. 
\end{cases}
\]
\end{Corollary}

\begin{proof}
In view of Proposition \ref{Prop : proper Beth completion : varphi is functional}, the pp formula $\varphi_{k, n}$ is functional in $\A_n$. By \cite[Cor.~3.11]{CKMIMPv3} it is also functional in $\QQQ(\A_n)$. In view of \cref{Cor : proper Beth completion : V(An) = Q(An)}, this means that $\varphi_{k, n}$ is functional in $\VVV(\A_n)$  and, therefore, defines an implicit operation $f_{k, n} \in \imppp(\VVV(\A_n))$. From \cref{Prop : proper Beth completion : varphi is functional} it follows that $f_{k, n}^{\A_n}$ is total and defined as in the statement. 
As $f_{k, n}^{\A_n}$ is total, we can apply \cite[Prop.~8.11(ii)]{CKMIMPv3} to the case where $\K = \VVV(\A_n) = \QQQ(\A_n)$ and $\mathsf{M} = \{ \A_n \}$, obtaining that $f_{k, n}$ is extendable. Thus, we conclude that $f_{k, n} \in \ext_{pp}(\VVV(\A_n))$.
\end{proof}

Now, for every $n \geq 3$ let 
\[
\mathcal{F}_n = \{ f_{k, n} : k \text{ is positive and }\leq n-1 \}.
\]
Observe that $\mathcal{F}_n \subseteq \ext_{pp}(\VVV(\A_n))$ by Corollary \ref{Cor : proper Beth completion : fnk is extendable}. Then consider an $\mathcal{F}_n$-expansion
\[
\mathscr{L}_{\mathcal{F}_n} = \mathscr{L} \cup \{ \ell_f : f \in \mathcal{F}_n \}
\]
of the language $\mathscr{L}$ of Heyting algebras and let
\[
\mathsf{B}(n) = \SSS(\VVV(\A_n)[\mathscr{L}_{\mathcal{F}_n}])
\]
be the corresponding pp expansion of $\VVV(\A_n)$. Our aim is to prove the following.

\begin{Theorem}\label{Thm : proper Beth completion : B(n) is the companion}
Let $n \geq 3$. Then $\mathsf{B}(n)$ is  a congruence preserving  Beth companion of $\VVV(\A_n)$. 
\end{Theorem}

To this end, recall from Corollary \ref{Cor : proper Beth completion : fnk is extendable} that $f^{\A_n}$ is total for every $f \in \mathcal{F}_n$, whence the algebra
\[
\B_n = \A_n[\mathscr{L}_{\mathcal{F}_n}]
\]
is  
defined. We begin with the following observation.

\begin{Proposition}\label{Prop : B(n) : proper Beth completion : arithmetical}
For every $n \geq 3$ we have 
\[
\mathsf{B}(n) = \VVV(\B_n)\, \, \text{ and } \, \,\mathsf{B}(n)_\textsc{fsi} = \III\SSS(\B_n).
\]
Moreover, $\mathsf{B}(n)$ is an arithmetical variety.
\end{Proposition}
\begin{proof}
We begin with the following observation.

\begin{Claim}\label{Claim : new claim to answer Luca's question}
We have $\VVV(\B_n)_\textsc{fsi} = \III\SSS(\B_n)$.
\end{Claim}

\begin{proof}[Proof of the Claim]
First, we show that 
\begin{equation}\label{Eq : addendum : answering Luca's question}
\mathsf{Con}(\C) = \mathsf{Con}(\C{\upharpoonright}_{\mathscr{L}})\text{  for every }\C \in \III\SSS(\B_n).
\end{equation}
Clearly, it will be enough to prove the above display for an arbitrary $\C \in \SSS(\B_n)$. The inclusion $\mathsf{Con}(\C) \subseteq \mathsf{Con}(\C{\upharpoonright}_{\mathscr{L}})$ is straightforward. To prove the reverse one, consider $\theta \in  \mathsf{Con}(\C{\upharpoonright}_{\mathscr{L}})$. From $\C \leq \B_n$ it follows that $\C{\upharpoonright}_{\mathscr{L}} \leq (\B_n){\upharpoonright}_{\mathscr{L}} = \A_n$. As $\C{\upharpoonright}_{\mathscr{L}}$ and $\A_n$ are Heyting algebras and the variety of Heyting algebras has the congruence extension property, there exists $\phi \in \mathsf{Con}(\A_n)$ such that $\theta = \phi{\upharpoonright}_C$. From   \cite[Prop.~12.10(ii)]{CKMIMPv3} 
and the definition of $\B_n$ it follows that $\mathsf{Con}(\A_n) = \mathsf{Con}(\B_n)$. Therefore, $\phi \in \mathsf{Con}(\B_n)$. Together with $\C \leq \B_n$, this yields $\theta = \phi{\upharpoonright}_C \in \mathsf{Con}(\C)$, as desired.

Next, we prove $\VVV(\B_n)_\textsc{fsi} = \III\SSS(\B_n)$. By \cref{Prop : Jonsson lattice : easy} we have $\VVV(\B_n)_\textsc{fsi} \subseteq \HHH\SSS(\B_n)$. Therefore, it suffices to show that the finitely subdirectly irreducible members of $\HHH\SSS(\B_n)$ are precisely the members of $\III\SSS(\B_n)$. To this end, consider a finitely subdirectly irreducible $\C \in \HHH\SSS(\B_n)$. Then there exist $\D \leq \B_n$ and $\theta \in \mathsf{Con}(\D)$ such that $\C \cong \D / \theta$. By \cite[Prop.~2.10]{CKMIMPv3} the congruence $\theta$ is meet irreducible in $\mathsf{Con}(\D)$. By (\ref{Eq : addendum : answering Luca's question}) it is also a meet irreducible member of $\mathsf{Con}(\D{\upharpoonright}_{\mathscr{L}})$. Since $\D{\upharpoonright}_{\mathscr{L}} \leq \A_n$, one can check by inspection that the only meet irreducible congruences of $\D{\upharpoonright}_{\mathscr{L}}$ are $\textup{id}_D$ and the congruences $\phi$ of $\D{\upharpoonright}_{\mathscr{L}}$ with exactly two equivalences, namely, $0 / \phi$ and $1 / \phi$. If $\theta = \textup{id}_D$, then $\C \cong \D$ and, therefore, $\C \in \III\SSS(\B_n)$ because $\D \leq \B_n$. On the other hand, if $\theta$ has exactly two equivalence classes $0 / \theta$ and $1 / \theta$, then $\D / \theta$ is isomorphic to the subalgebra of $\B_n$ with universe $\{ 0, 1 \}$, whence $\C \in \III \SSS(\B_n)$. Finally, we show that every member of $\III\SSS(\B_n)$ is finitely subdirectly irreducible. Let $\C \in \III\SSS(\B_n)$. Then $\mathsf{Con}(\C) = \mathsf{Con}(\C{\upharpoonright}_{\mathscr{L}})$ by (\ref{Eq : addendum : answering Luca's question}). Since $\C \in \III\SSS(\B_n)$, the definition of $\B_n$ guarantees that $\C{\upharpoonright}_{\mathscr{L}} \in \III\SSS(\A_n)$. By inspection one can check that every member of $\III\SSS(\A_n)$ is finitely subdirectly irreducible. Consequently, so is $\C{\upharpoonright}_{\mathscr{L}}$. By \cite[Prop.~2.10]{CKMIMPv3} the congruence $\textup{id}_C$ is meet irreducible in $\mathsf{Con}(\C{\upharpoonright}_{\mathscr{L}})$. As $\mathsf{Con}(\C) = \mathsf{Con}(\C{\upharpoonright}_{\mathscr{L}})$, it is also meet irreducible in $\mathsf{Con}(\C)$. Hence, we conclude that $\C$ is finitely subdirectly irreducible by \cite[Prop.~2.10]{CKMIMPv3}.
\end{proof}

By \cref{Claim : new claim to answer Luca's question} and the Subdirect Decomposition Theorem (see, e.g., \cite[Thm.\ 3.1.1]{Go98a}) we obtain $\VVV(\B_n) = \III\SSS\PPP(\VVV(\B_n)_\textsc{fsi})= \III\SSS\PPP\III\SSS(\B_n)$.\ 
Consequently, $\VVV(\B_n) \subseteq \QQQ(\B_n)$. As the reverse inclusion always holds, we conclude that $\VVV(\B_n) = \QQQ(\B_n)$.

Now, recall from \cref{Cor : proper Beth completion : V(An) = Q(An)} that $\VVV(\A_n) = \QQQ(\A_n)$. As $\B_n = \A_n[\mathscr{L}_{\mathcal{F}_n}]$, this allows us to apply \cite[Thm.~10.5]{CKMIMPv3} to the case where $\K = \VVV(\A_n)$,  $\mathsf{N} = \{ \A_n \}$, and $\mathbb{O}= \QQQ$, obtaining $\mathsf{B}(n) = \SSS(\VVV(\A_n)[\mathscr{L}_{\mathcal{F}_n}]) = \QQQ(\A_n[\mathscr{L}_{\mathcal{F}_n}]) = \QQQ(\B_n)$. Since $\QQQ(\B_n) = \VVV(\B_n)$, we obtain $\mathsf{B}(n) = \VVV(\B_n)$. Therefore, $\mathsf{B}(n)_\textsc{fsi} = \VVV(\B_n)_\textsc{fsi} = \III\SSS(\B_n)$. 
Lastly, since $\B_n$ has a Heyting algebra reduct, the variety $\VVV(\B_n)$ is arithmetical (see, e.g., \cite[p.~80]{BuSa00}).
\end{proof}

An \emph{endomorphism} of an algebra $\A$ is a homomorphism $h \colon \A \to \A$. When $h$ is an isomorphism, we say that it is an \emph{automorphism} of $\A$. The sets of endomorphisms and of automorphisms of $\A$ will be denoted, respectively, by $\mathsf{end}(\A)$ and $\mathsf{aut}(\A)$. 

Similarly to the case of complete atomic Boolean algebras (cf.\ \cite[Cor.~14.2]{BAGiHa}), one can easily verify that every permutation of the atoms of $\A_n$ for some $n \in \mathbb{N}$ 
induces an automorphism of $\A_n$ 
in the following way. 

\begin{Proposition}\label{Prop : proper Beth completion : the auto sigma}
Let $n \in \mathbb{N}$ and let $\sigma\colon  \mathsf{at}(\A_n) \to \mathsf{at}(\A_n)$ be a permutation. 
Then the map $\sigma^* \colon \A_n \to \A_n$ defined for every $a \in \A_n$ as 
\[
\sigma^*(a)= \begin{cases}
  1  & \text{if }a = 1; \\
 \bigvee \sigma[\mathsf{at}_{\A_n}(a)] & \text{if }a \ne 1
\end{cases}
\]
 is an automorphism of $\A_n$.
\end{Proposition}

We will also make use of the next observation on the automorphisms of $\B_n$.

\begin{Proposition}\label{Prop : proper Beth completion : automorphisms}
The following conditions hold for every $n \geq 3$:
\benroman
\item \label{item : proper Beth completion : automorphisms : 1} for every $\A \leq \B_n$ and $b \in B_n - (A \cup \{ e \})$ there exists $h \in \mathsf{aut}(\B_n)$ such that $b \ne h(b)$ and $a = h(a)$ for every $a \in A$;
\item\label{item : proper Beth completion : automorphisms : 2} for every pair of embeddings $g, h \colon \A \to \B_n$ there exists $i \in \mathsf{aut}(\B_n)$ such that $g = i \circ h$.
\eroman
\end{Proposition}

\begin{proof}
(\ref{item : proper Beth completion : automorphisms : 1}): Consider $\A \leq \B_n$ and $b \in B_n - (A \cup \{ e \})$. For every $a \in \mathsf{at}(\A)$ let
\[
X_a = \mathsf{at}_{\B_n}(a). 
\]
We will prove that $\{ X_a : a \in \mathsf{at}(\A) \}$ forms a partition of $\mathsf{at}(\B_n)$. 
As $\A \leq \B_n$, for every distinct $a, c \in \mathsf{at}(\A)$ we have $X_a \cap X_c = \emptyset$. Therefore, it only remains to show that for every $a \in \mathsf{at}(\B_n)$ there exists $c \in \mathsf{at}(\A)$ such that $a \in X_c$, i.e., $a \leq c$. Consider $a \in \mathsf{at}(\B_n)$.  We begin by showing that $e \leq \bigvee \mathsf{at}(\A)$. If $A = \{0,1\}$, we have  $1 \in \mathsf{at}(\A)$ and, therefore, $e \leq \bigvee \mathsf{at}(\A) = 1$. Then we consider the case where $A \ne \{ 0, 1 \}$. In this case, there exists $a \in A$ such that $0 < a < 1$. Observe that $\lnot a \in A$ and $\mathsf{at}_{\A}(a) \cup \mathsf{at}_{\A}(\lnot a) \subseteq \mathsf{at}(\A)$. Consequently,  using \eqref{Eq : tricks for An : Beth completion : 2} and \eqref{Eq : tricks for An : Beth completion : 4}, we obtain \[e = a \lor \lnot a = \bigvee \mathsf{at}_{\A}(a) \lor \bigvee \mathsf{at}_{\A}(\lnot a) \leq \bigvee \mathsf{at}(\A).\] 
Hence, we conclude that $e \leq \bigvee \mathsf{at}(\A)$, as desired. 
Therefore, $a \leq \bigvee \mathsf{at}(\A)$ because $a \in \mathsf{at}(\B_n)$ and every atom of $\B_n$ is below $e$. 
Since $a \in \mathsf{at}(\B_n)$, from  $a \leq \bigvee \mathsf{at}(\A)$ it follows that $a \leq c$ for some $c \in \mathsf{at}(\A)$. Hence, $\{ X_a : a \in \mathsf{at}(\A) \}$ forms a partition of $\mathsf{at}(\B_n)$, as desired.

Now, observe that $1 \in A$ because $\A \leq \B_n$. Together with the assumption that $b \notin A \cup \{ e \}$, this yields $b < e$. We will show that there exist $a \in \mathsf{at}(\A)$ and $c, d \in X_a$ such that $c \leq b$ and $d \nleq b$. We have two cases: either  $A = \{ 0, 1 \}$ or $A \ne \{ 0, 1 \}$. First, suppose that $A = \{ 0, 1 \}$. Then $\mathsf{at}(\A) = \{ 1 \}$ and $X_1 = \mathsf{at}(\B_n)$. Since $b < e$, there exist $c, d \in \mathsf{at}(\B_n) = X_1$ such that $c \leq b$ and $d \nleq b$, as desired.\ Next we consider the case where  $A \ne \{ 0, 1 \}$. 
 Recall from the first part of the proof that $\{X_a : a \in \mathsf{at}(\A)\}$ is a partition of $\mathsf{at}(\B_n)$. Therefore, $\mathsf{at}_{\B_n}(b) \subseteq \mathsf{at}(\B_n) = \bigcup \{X_a : a \in \mathsf{at}(\A)\}$. Suppose, with a view to contradiction, that for every $a \in \mathsf{at}(\A)$ we have $X_a \subseteq \mathsf{at}_{\B_n}(b)$ or $X_a \cap \mathsf{at}_{\B_n}(b) = \emptyset$. Then  
 \begin{equation} \label{Eq : at(b) in terms of X_a}
 \mathsf{at}_{\B_n}(b) = \bigcup \{X_a : a \in \mathsf{at}(\A) \text{ and } X_a \subseteq \mathsf{at}_{\B_n}(b)\}.
 \end{equation}
It follows that 
\begin{align*}
b &= \bigvee \mathsf{at}_{\B_n}(b)
= \bigvee \bigcup \{X_a : a \in \mathsf{at}(\A) \text{ and } X_a \subseteq \mathsf{at}_{\B_n}(b)\}\\ 
&= \bigvee\left\{\bigvee \mathsf{at}_{\B_n}(a) : a \in \mathsf{at}(\A) \text{ and }X_a \subseteq \mathsf{at}_{\B_n}(b)\right\}\\
&= \bigvee \{ a \in \mathsf{at}(\A) :  X_a \subseteq \mathsf{at}_{\B_n}(b)\},
\end{align*}
where the first equality holds by \eqref{Eq : tricks for An : Beth completion : 4} and $b \neq 1$ (the latter follows from $b \notin A$), the second by \eqref{Eq : at(b) in terms of X_a}, the third by the definition of $X_a$, and the fourth follows from \eqref{Eq : tricks for An : Beth completion : 4}  because $a \neq 1$ (the latter holds because $a \in \mathsf{at}(\A)$ and $A \neq \{0,1\}$).
But this is a contradiction to the assumption that $b \notin A$. 
Therefore, there exists $a \in \mathsf{at}(\A)$ such that $\emptyset \subsetneq X_a \cap \mathsf{at}_{\B_n}(b) \subsetneq X_a$. Consequently, we can choose $c \in X_a \cap \mathsf{at}_{\B_n}(b)$ to obtain $c \in X_a$ such that $c \leq b$ and $d \in X_a - \mathsf{at}_{\B_n}(b)$ such that $d \nleq b$. 
Thus, in both cases, there exist $a \in \mathsf{at}(\A)$ and $c, d \in B_n$ with
\begin{equation}\label{Eq : proper Beth completion : sigma fixes the partition : new}
c, d \in X_a, \qquad c \in \mathsf{at}_{\B_n}(b), \, \, \text{ and }\, \, d \nleq b.
\end{equation}

Then let $\sigma \colon \mathsf{at}(\B_n) \to \mathsf{at}(\B_n)$ be a permutation such that
\begin{equation}\label{Eq : proper Beth completion : sigma fixes the partition}
\sigma[X_a] = X_a\text{ for every }a \in \mathsf{at}(\A) \, \, \text{ and } \, \, \sigma(c) = d.
\end{equation}
Notice that $\sigma$ exists because $c, d \in X_a$ by the first item of (\ref{Eq : proper Beth completion : sigma fixes the partition : new}). 
 Recall that $\B_n = \A_n[\mathscr{L}_{\mathcal{F}_n}]$. Thus, we can consider the automorphism $\sigma^* \colon \A_n \to \A_n$ defined in  \cref{Prop : proper Beth completion : the auto sigma}, which by \cite[Prop.~9.5]{CKMIMPv3} is also an automorphism  of $\B_n$. 
Therefore, in order to complete the proof, it only remains to show that $\sigma^*(b) \ne b$ and $\sigma^*(a) = a$ for every $a \in A$. 

We begin by proving that
\[
\sigma^*(b) = \bigvee \sigma[\mathsf{at}_{\B_n}(b)] \geq \sigma(c) = d.
\]
The first step in the above display holds by the definition of $\sigma^*$ and $b < e < 1$, the second by the second  item of (\ref{Eq : proper Beth completion : sigma fixes the partition : new}), and the third by the right hand side of (\ref{Eq : proper Beth completion : sigma fixes the partition}). Together with the third item of (\ref{Eq : proper Beth completion : sigma fixes the partition : new}), the above display yields $\sigma^*(b) \ne b$.

Lastly, we will prove that $\sigma^*(a) = a$ for every $a \in A$. Consider $a \in A$. If $a = 1$, then $\sigma^*(a) = a$ by the definition of $\sigma^*$. Then we consider the case where $a \ne 1$. We will prove that
\begin{align*}
\sigma^*(a) &= \sigma^*\Big(\bigvee^{\A} \mathsf{at}_\A(a) \Big) = \sigma^*\Big(\bigvee^{\B_n}  \mathsf{at}_\A(a) \Big) =\bigvee^{\B_n}\sigma^*[\mathsf{at}_\A(a)] = \bigvee^{\B_n}_{p \in \mathsf{at}_\A(a)} \Big(\bigvee^{\B_n} \sigma[\mathsf{at}_{\B_n}(p)]\Big)\\
& = \bigvee^{\B_n}_{p \in \mathsf{at}_\A(a)} \Big(\bigvee^{\B_n} \mathsf{at}_{\B_n}(p)\Big) = \bigvee^{\B_n}  \mathsf{at}_\A(a) = \bigvee^{\A} \mathsf{at}_\A(a) = a.
\end{align*}
The above equalities are justified as follows: the first and the last hold by \eqref{Eq : tricks for An : Beth completion : 4} and $a \ne 1$,
the second and the second to last because $\A \leq \B_n$, the third because $\sigma^*$ is a homomorphism of bounded lattices and, therefore, it preserves finite (possibly empty) joins, the fourth by the definition of $\sigma^*$ and the fact that $p \leq a < 1$ implies $p \ne 1$, the fifth by the left hand side of (\ref{Eq : proper Beth completion : sigma fixes the partition}), and the sixth because $p \leq a < 1$ implies $p \leq e$, whence (\ref{Eq : tricks for An : Beth completion : 4}) guarantees that $p = \bigvee^{\B_n}\mathsf{at}_{\B_n}(p)$. 
Thus, we conclude that $\sigma^*(a) = a$ for every $a \in A$.

(\ref{item : proper Beth completion : automorphisms : 2}): Consider a pair of embeddings $g, h \colon \A \to \B_n$.  As $g$ and $h$ are homomorphisms of bounded lattices, we have $g(0) = h(0) = 0$ and $g(1) = h(1) = 1$. Therefore, if $A = \{ 0, 1 \}$, we have $g = h$ and we are done letting $i$ be the identity map on $B_n$.

Then we may assume that $A \ne \{ 0, 1 \}$, that is, $\{ 0, 1 \} \subsetneq A$. Since $g, h \colon \A \to \B_n$ are embeddings, both $g[\A]$ and $h[\A]$ are subalgebras of $\B_n$ containing at least an element $a$ other than $0$ and $1$. Then they must also contain $\lnot a$ and, therefore,   $e = a \lor \lnot a \in g[\A] \cap h[\A]$ by \eqref{Eq : tricks for An : Beth completion : 2}. As $e$ is the second largest element of $\B_n$ and $g$ and $h$ are embeddings of lattices, we obtain that $\A$ possesses a second largest element $e^*$ such that $g(e^*) = h(e^*) = e$. Moreover, $0 < e^* < 1$ because $e^*$ is the second largest element to $\A$ and $A \ne \{ 0, 1 \}$. If $A = \{ 0, e^*, 1 \}$, we have $g = h$ and we are done letting $i$ be the identity map on $B_n$. 

Then we may assume that $A \ne \{ 0, e^*, 1 \}$, that is, $\{ 0, e^*, 1 \} \subsetneq A$. We rely on the following series of observations.

\begin{Claim}\label{Claim : proper Beth completion : automorphism : 1} We have 
$g[\mathsf{at}(\A)] \cup h[\mathsf{at}(\A)] \subseteq \{ a \in B_n : 0 < a < e \}$.
\end{Claim}

\begin{proof}[Proof of the Claim]

By symmetry it suffices to show that $g[\mathsf{at}(\A)] \subseteq \{ a \in B_n : 0 < a < e \}$. To this end, consider $a \in \mathsf{at}(\A)$. Then $a > 0$. Moreover, since $e^*$ is the second largest element of $\A$ and $\A$ contains an element other than $0, e^*$, and $1$, from $a \in \mathsf{at}(\A)$ it follows that $a < e^*$. Therefore, $0 < a < e^*$. Since $g$ is a  embedding of bounded lattices, we obtain $0 = g(0) < g(a) < g(e^*)$. As we already established $g(e^*) = e$, we conclude that $0 < g(a) < e$.
\end{proof}

\begin{Claim}\label{Claim : proper Beth completion : automorphism : 2}
For every $a \in \mathsf{at}(\A)$ we have $\vert \mathsf{at}_{\B_n}(g(a))\vert  = \vert \mathsf{at}_{\B_n}(h(a))\vert$.
\end{Claim}

\begin{proof}[Proof of the Claim]
Recall that $\A_n$ has $n$ atoms by definition. As $\B_n$ is an expansion of $\A_n$, we obtain that $\B_n$ has $n$ atoms as well. Then consider $a \in B_n - \{ 0, e, 1 \}$  and observe that $|\mathsf{at}_{\B_n}(a)| \leq n-1$ because $|\mathsf{at}_{\B_n}(a)| = n$ by \eqref{Eq : tricks for An : Beth completion : 4} would imply $a \geq e$. 
Recall that $\mathscr{L}_{\mathcal{F}_n} = \mathscr{L} \cup \{ \ell_f : f \in \mathcal{F}_n\}$. Therefore, from \cref{Cor : proper Beth completion : fnk is extendable} and $\ell_{f_{k, n}}^{\B_n} = f_{k, n}^{\A_n}$ it follows that for every $m \leq n-1$,
\begin{equation}\label{Eq : proper Beth completion : automorphisms : 2}
\vert \mathsf{at}_{\B_n}(a)\vert = m \iff\text{for every }
0 < k \leq n-1 \text{ we have } \ell_{f_{k, n}}^{\B_n}(a)= \begin{cases}
  1  & \text{if }  m \leq k;\\ 
 e & \text{if }  m \geq k+1. 
\end{cases}
\end{equation}

To prove the statement of the claim, consider $a \in \mathsf{at}(\A)$. By Claim 
\ref{Claim : proper Beth completion : automorphism : 1} 
we have $0 < g(a), h(a) < e$. Then $\vert \mathsf{at}_{\B_n}(g(a)) \vert$ is a positive integer  $m \leq n-1$. 
In view of (\ref{Eq : proper Beth completion : automorphisms : 2}), for every positive $k \leq n-1$,
\[
\ell_{f_{k, n}}^{\B_n}(g(a))= \begin{cases}
  1  & \text{if } m \leq k;\\ 
 e & \text{if }  m \geq k+1. 
\end{cases}
\]
Since $g \colon \A \to \B_n$ is an embedding such that $g(e^*) = e$ and $g(1) = 1$, this yields that for every positive $k \leq n-1$,
\[
\ell_{f_{k, n}}^{\A}(a)= \begin{cases}
  1  & \text{if } m \leq k;\\ 
 e^* & \text{if }  m \geq k+1. 
\end{cases}
\]
As $h \colon \A \to \B_n$ is also an embedding such that $h(e^*) = e$ and $h(1) = 1$, we obtain that for every positive $k \leq n-1$,
\[
\ell_{f_{k, n}}^{\B_n}(h(a))= \begin{cases}
  1  & \text{if } m \leq k;\\ 
 e & \text{if }  m \geq k+1. 
\end{cases}
\]
Together with (\ref{Eq : proper Beth completion : automorphisms : 2}), this yields $\vert \mathsf{at}_{\B_n}(h(a)) \vert = m$.
\end{proof}

\begin{Claim}\label{Claim : proper Beth completion : automorphism : 3}
For every $a, b \in \mathsf{at}(\A)$,
\[
\text{if }a \ne b \text{, then }\mathsf{at}_{\B_n}(g(a)) \cap \mathsf{at}_{\B_n}(g(b)) = \emptyset = \mathsf{at}_{\B_n}(h(a)) \cap \mathsf{at}_{\B_n}(h(b)).
\]
\end{Claim}

\begin{proof}[Proof of the Claim]
Suppose that $a \ne b$. By symmetry it suffices to show that $\mathsf{at}_{\B_n}(g(a)) \cap \mathsf{at}_{\B_n}(g(b)) = \emptyset$. From $a \ne b$ and $a, b \in \mathsf{at}(\A)$ it follows that $a \land^{\A} b = 0$. Consequently, $g(a) \land^{\B_n} g(b) = 0$ because $g \colon \A \to \B_n$ is an embedding. Therefore, we conclude that $\mathsf{at}_{\B_n}(g(a)) \cap \mathsf{at}_{\B_n}(g(b)) = \emptyset$.
\end{proof}

In view of Claims \ref{Claim : proper Beth completion : automorphism : 2} and \ref{Claim : proper Beth completion : automorphism : 3} there exists a permutation $\sigma \colon \mathsf{at}(\B_n) \to \mathsf{at}(\B_n)$ such that 
\begin{equation}\label{Eq : proper Beth completion : automorphisms : 3}
\sigma[\mathsf{at}_{\B_n}(h(a))] = \mathsf{at}_{\B_n}(g(a))\text{ for every }a \in \mathsf{at}(\A).
\end{equation}
 As $\B_n = \A_n[\mathscr{L}_{\mathcal{F}_n}]$, the map $\sigma$ can also be viewed as a permutation of $\mathsf{at}(\A_n)$. 
Consequently, \cref{Prop : proper Beth completion : the auto sigma} yields an automorphism $\sigma^* \colon \A_n \to \A_n$, which by \cite[Prop.~9.5]{CKMIMPv3} is also an automorphism of $\B_n$.  
To conclude the proof, it only remains to show that $g = \sigma^* \circ h$, for in this case we can take $i = \sigma^*$.

From the assumption that $g, h$, and $\sigma^*$ are homomorphisms of bounded lattices it follows that $g(1) = h(1) = \sigma^*(1) = 1$, whence $g(1) = \sigma^* (h(1))$. Therefore, it suffices to show that $g(a) = \sigma^*(h(a))$ for every $a \in A - \{ 1 \}$. We will prove that for every $a \in A - \{ 1 \}$,
\begin{align*}
g(a) &= g\Big(\bigvee^\A \mathsf{at}_{\A}(a)\Big) = \bigvee^{\B_n} g[\mathsf{at}_\A(a)] = \bigvee_{b \in \mathsf{at}_\A(a)}^{\B_n}\bigvee^{\B_n}\mathsf{at}_{\B_n}(g(b)) = \bigvee_{b \in \mathsf{at}_\A(a)}^{\B_n} \bigvee^{\B_n}\sigma[\mathsf{at}_{\B_n}(h(b))]\\
& = \bigvee_{b \in \mathsf{at}_\A(a)}^{\B_n} \sigma^*(h(b)) = \sigma^*\Big(h\Big(\bigvee^{\A} \mathsf{at}_{\A}(a)\Big)\Big) = \sigma^*(h(a)).
\end{align*}
The above equalities are justified as follows. The first and the last hold by \eqref{Eq : tricks for An : Beth completion : 4} 
and the assumption that $a \ne 1$, the second and the second to last because $g, h$, and $\sigma^*$ preserve finite (possibly empty) joins because they are homomorphisms of bounded lattices, the third by Claim 
\ref{Claim : proper Beth completion : automorphism : 1} and (\ref{Eq : tricks for An : Beth completion : 4}), the fourth by (\ref{Eq : proper Beth completion : automorphisms : 3}),  and the fifth follows from \cref{Claim : proper Beth completion : automorphism : 1} and the definition of $\sigma^*$.
  Hence, we conclude that $g = \sigma^* \circ h$.
\end{proof}

 Finalizing the proof of the fact that $\mathsf{B}(n)$ is a congruence preserving Beth companion of $\VVV(\A_n)$ (\cref{Thm : proper Beth completion : B(n) is the companion}) 
requires some further investigation of the variety $\mathsf{B}(n)$ and its properties. While $\VVV(\A_n)$ lacks the amalgamation property for every $n \geq 3$ (see \cite[Thm.~2]{Mak77}), this property holds in the pp expansion $\mathsf{B}(n)$ of $\VVV(\A_n)$, as we proceed to illustrate.
 To this end, we will employ the following 
result \cite[Thm.\ 3.4]{FMfsi}\footnote{Our formulation of Theorem \ref{Thm : AP : George Wesley} is slightly different from the one of \cite[Thm.\ 3.4]{FMfsi}. However, the difference is insubstantial and amounts to the fact that in   \cite{FMfsi} the class $\mathsf{K}_\textsc{rfsi}$ is defined as $\K_\textsc{rfsi}^*$.} (see also \cite[Thm.\ 3]{GLAP}), together with the observation that $\mathsf{B}(n)$ has the congruence extension property for each $n \geq 3$. 

Given a quasivariety $\mathsf{K}$, let
\[
\mathsf{K}_\textsc{rfsi}^* = \mathsf{K}_\textsc{rfsi} \cup \{ \A \in \K : \A \text{ is trivial} \}.
\]

\begin{Theorem}\label{Thm : AP : George Wesley}
Let $\K$ be a quasivariety with the relative congruence extension property such that $\mathsf{K}_\textsc{rfsi}$ is closed under nontrivial subalgebras. Then $\mathsf{K}$ has the amalgamation property if and only if $\mathsf{K}_\textsc{rfsi}^*$ has the amalgamation property.
\end{Theorem}

To show that $\mathsf{B}(n)$ has the congruence extension property for each $n \geq 3$, we rely on the following preservation result.

\begin{Proposition} \label{Prop : CEP preserved in pp exp}
Let $\M$ be a pp expansion of a quasivariety $\K$. If $\mathsf{K}$ has the relative congruence extension property, then so does $\mathsf{M}$.
\end{Proposition}

\begin{proof}
Suppose that $\K$ has the relative congruence extension property. Then consider $\A \leq \B \in \M$ and $\theta \in \mathsf{Con}_\M(\A)$. We need to find some $\phi\in\mathsf{Con}_\M(\B)$ such that $\theta = \phi{\upharpoonright}_A$. Since $\A \in \SSS(\M) = \M$, from  \cite[Thm.~12.10(i)]{CKMIMPv3} 
it follows that $\mathsf{Con}_\M(\A) \subseteq \mathsf{Con}_\K(\A{\upharpoonright}_{\mathscr{L}_\K})$, whence  $\theta \in \mathsf{Con}_\K(\A{\upharpoonright}_{\mathscr{L}_\K})$. Since $\M$ is a pp expansion of $\K$, it is of the form $\SSS(\K[\mathscr{L}_\mathcal{F}])$. Together with $\A \leq \B \in \M$, this implies $\A \leq \B \leq \C$ for some $\C \in \K[\mathscr{L}_\mathcal{F}]$. Consequently, $\A{\upharpoonright}_{\mathscr{L}_\K} \leq \C{\upharpoonright}_{\mathscr{L}_\K} \in \K$. As $\theta \in \mathsf{Con}_\K(\A{\upharpoonright}_{\mathscr{L}_\K})$ and $\K$ has the relative congruence extension property by assumption, there exists $\eta\in\mathsf{Con}_\K(\C{\upharpoonright}_{\mathscr{L}_\K})$ such that $\theta = \eta{\upharpoonright}_A$. Recall from  \cite[Thm.~12.10(ii)]{CKMIMPv3} 
that $\C \in \K[\mathscr{L}_\mathcal{F}]$ implies $\mathsf{Con}_\M(\C) = \mathsf{Con}_\K(\C{\upharpoonright}_{\mathscr{L}_\K})$, whence $\eta \in \mathsf{Con}_\M(\C)$. This yields $\eta{\upharpoonright}_B \in \mathsf{Con}_\M(\B)$ and $\theta= (\eta{\upharpoonright}_B){\upharpoonright}_A$ because $\A \leq \B \leq \C$ and $\theta = \eta{\upharpoonright}_A$. Hence, we are done letting $\phi = \eta{\upharpoonright}_B$.
\end{proof}

\begin{Proposition}\label{Prop : proper Beth completion : B(n) has only the CEP}
For every $n \geq 3$ the variety $\mathsf{B}(n)$ has the congruence extension property.
\end{Proposition}

\begin{proof}
We recall that every variety of Heyting algebras has the congruence extension property.
In particular, $\VVV(\A_n)$ has the congruence extension property for every $n \geq 3$. 
Therefore, \cref{Prop : CEP preserved in pp exp}  yields that the pp expansion $\mathsf{B}(n)$ of $\VVV(\A_n)$  has the congruence extension property.
\end{proof}

\begin{Proposition}\label{Prop : B(n) : proper Beth completion : AP and CEP : final}
For every $n \geq 3$ the variety $\mathsf{B}(n)$ has the amalgamation property.
\end{Proposition}

\begin{proof}
Recall from \cref{Prop : proper Beth completion : B(n) has only the CEP} that the variety $\mathsf{B}(n)$ has the congruence extension property. Moreover, $\mathsf{B}(n)_\textsc{fsi}$ is closed under subalgebras by \cref{Prop : B(n) : proper Beth completion : arithmetical}.\ Therefore,  in view of \cref{Thm : AP : George Wesley},  in order to prove that $\mathsf{B}(n)$ has the amalgamation 
property, it suffices to show that $\mathsf{B}(n)_\textsc{fsi}^*$ has the amalgamation property. To this end, consider a pair of embeddings $h_1 \colon \A \to \B$ and $h_2 \colon \A \to \C$ with $\A, \B, \C \in \mathsf{B}(n)_\textsc{fsi}^*$. We need to find a pair of embeddings $g_1 \colon \B \to \D$ and $g_2 \colon \C \to \D$ with $\D \in \mathsf{B}(n)_\textsc{fsi}^*$ such that $g_1 \circ h_1 = g_2 \circ h_2$.

We have two cases depending on whether $\A$ is trivial or nontrivial.
First, suppose that $\A$ is trivial. As $\mathsf{B}(n)_\textsc{fsi}$ is closed under subalgebras by Proposition \ref{Prop : B(n) : proper Beth completion : arithmetical} and finitely subdirectly irreducible algebras are nontrivial, we obtain that no member of $\mathsf{B}(n)_\textsc{fsi}$ has a trivial subalgebra.
Since $\A$ embeds into $\B$ and $\C$, this yields $\B, \C \notin \mathsf{B}(n)_\textsc{fsi}$. Therefore, $\B$ and $\C$ are trivial because $\B, \C \in \mathsf{B}(n)_\textsc{fsi}^*$. Consequently, $\A, \B$, and $\C$ are all trivial and the embeddings $h_1 \colon \A \to \B$ and $h_2 \colon \A \to \C$ are isomorphisms. Therefore, we may assume that $\A = \B = \C$ and that $h_1$ and $h_2$ are the identity map $i$ on $A$. 
Hence, letting $\D = \A$ and $g_1 = g_2 = i$, we are done. 

Next we consider the case where $\A$ is nontrivial. Since $\A$ embeds into $\B$ and $\C$, we obtain that $\B$ and $\C$ are also nontrivial. Together with $\B, \C \in \mathsf{B}(n)_\textsc{fsi}^*$, this yields $\B, \C \in \mathsf{B}(n)_\textsc{fsi}$. Recall from \cref{Prop : B(n) : proper Beth completion : arithmetical} that $\mathsf{B}(n)_\textsc{fsi} = \III\SSS(\B_n)$, whence $\B, \C \in \III\SSS(\B_n)$. Therefore, we may assume that $\B = \C = \B_n$ and that $h_1$ and $h_2$ are embeddings of $\A$ into $\B_n$. By Proposition \ref{Prop : proper Beth completion : automorphisms}(\ref{item : proper Beth completion : automorphisms : 2}) there exists $i \in \mathsf{aut}(\B_n)$ such that $h_1 = i \circ h_2$. Let $\D = \B_n$, $g_2 = i$, and $g_1$ the identity map on $B_n$. Clearly, $g_1, g_2 \colon \B_n \to \B_n$ are embeddings such that $g_1 \circ h_1 = h_1 = i \circ h_2 = g_2 \circ h_2$.
\end{proof}

We are now ready to prove \cref{Thm : proper Beth completion : B(n) is the companion}.

\begin{proof}
Recall  that
$\mathsf{B}(n)$ is a pp expansion of $\VVV(\A_n)$. Moreover, since $\mathbb{V}(\A_n)$ has the congruence extension property, 
we can apply  \cite[Thm.~12.3]{CKMIMPv3}, 
obtaining  that $\mathsf{B}(n)$ is congruence preserving. 
Hence,  by \cite[Thm.~11.6]{CKMIMPv3} it will be enough to prove that $\mathsf{B}(n)$ has the strong epimorphism surjectivity property. Recall from Propositions \ref{Prop : B(n) : proper Beth completion : arithmetical} and \ref{Prop : B(n) : proper Beth completion : AP and CEP : final} that $\mathsf{B}(n)$ is an arithmetical variety with the amalgamation property. Therefore, in view of \cite[Cor.~7.16]{CKMIMPv3}, it will be enough to show that every $\C \in \mathsf{B}(n)_\textsc{fsi}$ lacks proper $\mathsf{B}(n)$-epic subalgebras. To this end, consider $\A \leq \C \in \mathsf{B}(n)_\textsc{fsi}$ with $\A \leq \C$ proper. Then there exists $b \in C - A$.  Moreover, we may assume that $\C \leq \B_n$ by Proposition \ref{Prop : B(n) : proper Beth completion : arithmetical}, whence $\A \leq \C  \leq \B_n$.

Let $i$ be the identity map on $\B_n$. As $i \in \mathsf{end}(\B_n)$ and $b \in C$, to conclude the proof, it will be enough to find some $h \in \mathsf{end}(\B_n)$ such that $h{\upharpoonright}_A = i{\upharpoonright}_A$ and $h(b) \ne i(b)$. For, by considering the restrictions of $h$ and $i$ to $\C \leq \B_n$, we obtain that $\A \leq \C$ is not $\mathsf{B}(n)$-epic, as desired.

We have two cases: either $e \notin A$ or $e \in A$. First, suppose that $e \notin A$. 
Since $\A \leq \B_n$, we have $\A\res_{\L} \leq (\B_n)\res_{\L}=\A_n$. Together with $e \notin A$ and \eqref{Eq : tricks for An : Beth completion : 2}, this yields $A = \{ 0, 1 \}$. Then $0 < b$ because $b \notin A$. Let $a \in \mathsf{at}_{\B_n}(b)$ and consider the map $h \colon \B_n \to \B_n$ defined for every $c \in B_n$ as
\[
h(c)= \begin{cases}
  1  & \text{if }a \leq c; \\
0 & \text{if }a \nleq c.
\end{cases}
\]
Since $h \in \mathsf{end}(\A_n)$  and $\B_n = \A_n[\mathscr{L}_{\mathcal{F}_n}]$, from \cite[Prop.~9.5]{CKMIMPv3}  
it follows that $h \in \mathsf{end}(\B_n)$.
Moreover, $a \in \mathsf{at}_{\B_n}(b)$ and the definition of $h$ imply $h(b) = 1$. Then $h(b) \ne b$ because  $b \notin A = \{ 0, 1 \}$. Thus, 
$h, i \colon \B_n \to \B_n$ are homomorphisms such that $h(b) \ne b = i(b)$ and $h{\upharpoonright}_A = i{\upharpoonright}_A$ (the latter because $A = \{ 0, 1 \}$ and both 
 $h$
and $i$ preserve the constants $0$ and $1$).

Lastly, we consider the case where $e \in A$. As $\A \leq \C$ is proper and $\C \leq \B_n$, there exists $b \in C - (A \cup \{ e \}) \subseteq B_n - (A \cup \{ e \})$.
By Proposition \ref{Prop : proper Beth completion : automorphisms}(\ref{item : proper Beth completion : automorphisms : 1}) there also exists $h \in \mathsf{aut}(\B_n)$ such that $b \ne h(b)$ and $a = h(a)$ for every $a \in A$. Thus, $h, i \colon \B_n \to \B_n$ are homomorphisms such that $h{\upharpoonright}_A = i{\upharpoonright}_A$ and $h(b) \ne b = i(b)$.
\end{proof}

Lastly, we prove Theorem \ref{Thm : proper Beth companion : main}. Notice that this concludes the proof of Theorem \ref{Thm : proper Beth companion}.

\begin{proof}
 Recall from \cref{Thm : proper Beth completion : B(n) is the companion} that $\mathsf{B}(n)$ is a congruence preserving Beth companion of $\mathbb{V}(\A_n)$.   
Therefore, it remains to show that $\mathbb{V}(\A_n)$ has no simple Beth companion.
Suppose the contrary, with a view to contradiction. Then let $a$ be an atom of $\B_n$ and consider $\C = \textsf{Sg}^{\B_n}(a)$. The following is an immediate consequence of the definition of $\C$.

\begin{Claim}\label{Claim : proper Beth completion : term equivalence : 0 claim}
The universe of $\C$ is $\{ 0, a, \lnot a, e, 1 \}$. Moreover, the Heyting algebra reduct of $\C$ is isomorphic to $\A_2$ 
with minimum $0$, maximum $1$, second largest element $e$, and atoms $a$ and $\lnot a$.
\end{Claim}

 As $a$ is an atom of $\B_n$ and $\A_n$ shares its bounded lattice reduct with $\B_n$, the number of atoms of $\A_n$ below $a$ is $1$. Since $\A_n$ has $n \geq 3$ atoms, from (\ref{Eq : tricks for An : Beth completion : 5}) it follows that the number of atoms of $\A_n$ below $\lnot a$ is $n-1 \geq 3-1 \geq 2$. Therefore, from \cref{Cor : proper Beth completion : fnk is extendable} it follows that $\ell_{f_{1, n}}^{\B_n}(a) = 1$ and $\ell_{f_{1, n}}^{\B_n}(\lnot a) = e$. As $\C \leq \B_n$, we obtain
\begin{equation}\label{Eq : proper Beth completion : f on A is asymmetric : 1}
\ell_{f_{1, n}}^{\?\C}(a) = 1 \, \, \text{ and } \, \, \ell_{f_{1, n}}^{\?\C}(\lnot a) = e.
\end{equation}

Recall from the assumptions that  $\mathbb{V}(\A_n)$ has a simple Beth companion.
By \cite[Thm.~11.7]{CKMIMPv3} all Beth companions of $\mathbb{V}(\A_n)$ are faithfully term equivalent relative to $\mathbb{V}(\A_n)$.  
Since $\mathsf{B}(n)$ is a Beth companion of $\mathbb{V}(\A_n)$, we conclude that $\mathsf{B}(n)$ is faithfully term equivalent relative to $\mathbb{V}(\A_n)$ to a Beth companion of $\mathbb{V}(\A_n)$ of the form $\mathbb{V}(\A_n)[\mathscr{L}_{\mathcal{F}^*}]$ for some $\mathcal{F}^* \subseteq \extpp(\mathsf{K})$. 
Furthermore, recall that $\mathsf{B}(n)$ is a variety by \cref{Prop : B(n) : proper Beth completion : arithmetical}. Therefore, from \cite[Rem.~11.12(v)]{CKMIMPv3} it follows that the class $\VVV(\A_n)[\mathscr{L}^*_\mathcal{F}]$ is also a variety.

Let $\tau$ and $\rho$ be the maps witnessing the fact that $\mathsf{B}(n)$ and $\mathbb{V}(\A_n)[\mathscr{L}_{\mathcal{F}^*}]$
are faithfully term equivalent relative to $\VVV(\A_n)$. 
We may assume that for every $\D \in \mathsf{B}(n)$,
\begin{equation}\label{Eq : proper Beth completion : f on A is asymmetric : 23}
\tau(\D) \in  \mathbb{V}(\A_n)[\mathscr{L}_{\mathcal{F}^*}] \, \, \text{ and } \, \, \D{\upharpoonright}_{\mathscr{L}} = \tau(\D){\upharpoonright}_{\mathscr{L}}.
\end{equation}
As $\C \leq \B_n \in \mathsf{B}(n)$ and $\mathsf{B}(n)$ is a variety
by \cref{Prop : B(n) : proper Beth completion : arithmetical}, we have $\C \in \mathsf{B}(n)$. Then $\tau(\C) \in  \mathbb{V}(\A_n)[\mathscr{L}_{\mathcal{F}^*}]$ by the left hand side of (\ref{Eq : proper Beth completion : f on A is asymmetric : 23}). Consequently, there exists $\D \in \VVV(\A_n)$ such that $\D[\mathscr{L}_{\mathcal{F}^*}]$ is  defined 
and $\tau(\C) =  \D[\mathscr{L}_{\mathcal{F}^*}]$. Together with the right hand side of (\ref{Eq : proper Beth completion : f on A is asymmetric : 23}), this yields
\[
\C{\upharpoonright}_\mathscr{L} = \tau(\C){\upharpoonright}_\mathscr{L} =  \D[\mathscr{L}_{\mathcal{F}^*}] {\upharpoonright}_\mathscr{L} = \D.
\]

In view of the above display, $\D$ is the Heyting algebra reduct of $\C$  and, therefore, is isomorphic to $\A_2$ with atoms $a$ and $\lnot a$  by \cref{Claim : proper Beth completion : term equivalence : 0 claim}. 
  This allows us to apply \cref{Prop : proper Beth completion : the auto sigma} 
to the permutation  $\sigma \colon \mathsf{at}(\D) \to \mathsf{at}(\D)$ that switches $a$ and $\lnot a$, thus obtaining  an automorphism $\sigma^* \colon \D \to \D$  
with
\begin{equation}\label{Eq : proper Beth completion : f on A is asymmetric : 2}
\sigma^*(a) = \lnot a \, \, \text{ and } \, \, \sigma^*(1) = 1.
\end{equation}
Moreover, from $\tau(\C) =  \D[\mathscr{L}_{\mathcal{F}^*}]$ it 
follows that $\C = \rho\tau(\C) = \rho( \D[\mathscr{L}_{\mathcal{F}^*}])$. Together 
with (\ref{Eq : proper Beth completion : f on A is asymmetric : 1}), this implies
\begin{equation}\label{Eq : proper Beth completion : f on A is asymmetric : 3}
\rho(\ell_{f_{1, n}})^{ \D[\mathscr{L}_{\mathcal{F}^*}]}(a) = 1 \, \, \text{ and } \, \, \rho(\ell_{f_{1, n}})^{ \D[\mathscr{L}_{\mathcal{F}^*}]}(\lnot a) = e.
\end{equation}
Recall from  \cite[Prop.~10.22(ii)]{CKMIMPv3} that there exists $g \in \mathsf{ext}_{pp}(\VVV(\A_n))$ such that 
\begin{equation} \label{Eq : g = rho-ell-f}
\rho(\ell_{f_{1, n}})^{ \D[\mathscr{L}_{\mathcal{F}^*}]} = g^{\D}.
\end{equation}
Together with the left hand side of (\ref{Eq : proper Beth completion : f on A is asymmetric : 3}), this yields $g^\D(a) = 1$. As the implicit operation $g$ is preserved by homomorphisms, we can apply the automorphism   $\sigma^*$  of $\D$ in (\ref{Eq : proper Beth completion : f on A is asymmetric : 2}) to deduce  
\[g^{\D}(\lnot a) = g^{\D}(\sigma^*(a)) = \sigma^*(g^{\D}(a)) = \sigma^*(1) = 1\]
and, therefore,
$\rho(\ell_{f_{1, n}})^{ \D[\mathscr{L}_{\mathcal{F}^*}]}(\lnot a) =1$ 
by \eqref{Eq : g = rho-ell-f}. Since $1 \ne e$, this contradicts the right hand side of (\ref{Eq : proper Beth completion : f on A is asymmetric : 3}).  Hence, we conclude that  $\VVV(\A_n)$ has a congruence preserving Beth companion (see \cref{Thm : proper Beth completion : B(n) is the companion}) but lacks a simple Beth companion.
\end{proof}

\end{document}